\documentclass[10pt]{amsart}
\usepackage{undertilde}
\usepackage{amssymb}
\usepackage{amsmath}
\usepackage{url}
\usepackage{latexsym}
\usepackage{mathrsfs}
\usepackage{amsthm}
\usepackage{verbatim}
\usepackage{graphicx}
\usepackage{epstopdf}
\usepackage{eqnarray}
\usepackage{epsfig}
\usepackage{color}
\usepackage{subfigure}
\usepackage{float}
\usepackage{extarrows}
\usepackage{color}
\usepackage{bm}
\usepackage[colorlinks,linkcolor=blue,anchorcolor=green,citecolor=red]{hyperref}
\usepackage{amsfonts}
\usepackage{fancyhdr}
\usepackage{extarrows}
\usepackage{amscd}
\usepackage{caption}
\usepackage{diagrams}
\usepackage{geometry}
\numberwithin{equation}{section}

\newtheorem{theorem}{Theorem}[section]
\newtheorem{remark}{Remark}

\newtheorem{corollary}{Corollary}[section]

\newcommand{\A}{{\mathbb A}}

\newcommand{\B}{{\mathbb B}}
\newcommand{\M}{{\mathbb M}}
\newcommand\R{\mathbb{R}}

\newcommand{\tr}{{\operatorname{tr}}}
\newcommand{\diag}{{\operatorname{diag}}}
\newcommand{\<}{\langle}
\renewcommand{\>}{\rangle}

\newcommand{\0}{\mathaccent23}

\begin{document}

\markboth{K. Hu AND R. Winther}{Template for Journal Of Computational
Mathematics}

\title{Well-conditioned frames for high order finite element methods}

\author{Kaibo Hu}
\address{School of Mathematics,
University of Minnesota, 55455 Minneapolis, MN, USA}
\email{khu@umn.edu}
\urladdr{http://www-users.math.umn.edu/~khu/}

\author{Ragnar Winther}
\address{Department of Mathematics,
University of Oslo, 0316 Oslo, Norway}
\email{rwinther@math.uio.no}
\urladdr{http://www.mn.uio.no/math/personer/vit/rwinther/index.html}

\maketitle

\begin{abstract}
The purpose of this paper is to discuss representations of high order $C^0$ finite element spaces on simplicial meshes in any dimension. 
When computing with high order piecewise polynomials  the conditioning 
of the basis  is likely to be important. The main result of this paper is a construction
of representations  by frames such that the associated $L^2$ condition number is bounded independently of 
the polynomial degree. To our knowledge, such  a representation has not been presented earlier.
The main tools we will use for the construction is 
the bubble transform, introduced previously in 
\cite{falk2013bubble}, and properties of Jacobi polynomials on simplexes in higher dimensions.
We also include a brief discussion of preconditioned iterative methods for the finite element systems in the setting of representations by frames.
\end{abstract}


\section{Introduction}\label{intro}
The discussion in this paper is motivated by finite element  discretizations of
second order elliptic equations, where $C^0$ piecewise polynomial spaces of high polynomial degree
are used as the finite dimensional space.
As the polynomial degree increases the choice of basis can have a substantial effect on the conditioning of the linear systems to be solved. 
The purpose of this paper is to discuss how to obtain representations of the finite element spaces which are uniformly well-conditioned
with respect to the polynomial degree. Here the conditioning of the representation is measured by the $L^2$ condition number.
Furthermore, we will explain how this 
influences the conditioning of the corresponding discrete systems.
 Since our main goal is to discuss dependence with respect to the polynomial degree
we will consider the mesh $\mathcal{T}_h$ to be fixed throughout the discussion below.

To motivate the discussion below, we consider a second order elliptic equation, defined on a bounded domain 
$\Omega \in \R^d$, which admits a weak formulation of the form: 

Find $u \in H^1(\Omega)$ such that
\begin{equation}\label{weak-1}
a(u,v) = f(v), \quad v \in H^1(\Omega),
\end{equation}
where $H^1(\Omega)$ denotes the Sobolev space of all functions in $L^2$ which also have all first order partial derivates in $L^2$. Furthermore, $f$ is a bounded linear functional,
and $a$ is a symmetric, bounded, and coercive bilinear form on $H^1(\Omega)$.
The formulation above reflects that we are considering an elliptic problem with natural boundary condition.
If we instead consider problems with  an essential boundary condition on parts of the boundary, we will obtain  a weak formulation with respect to a corresponding subspace of $H^1(\Omega)$. However, the effect of such modifications of \eqref{weak-1} will have minor effects on the discussion below. Therefore, we will restrict the discussion  to problems of the form \eqref{weak-1} throughout this paper.

A discretization of the problem \eqref{weak-1} can be derived from  a finite dimensional subspace 
$V_h$ of $H^1(\Omega)$. In the finite element method $V_h$ is typically a space of piecewise polynomials 
with respect to a partition, or a mesh, $\mathcal T_h$, with global $C^0$ continuity, and where the mesh parameter $h$ indicates the size of the cells of the partition. 
The corresponding discrete solution is defined by:

Find $u_h \in V_h$ such that
\begin{equation}\label{weak-h-1}
a(u_h,v) = f(v), \quad v \in V_h.
\end{equation}
This system can alternatively be written as a linear system of the form ${\mathcal{A}_h} u_h = f_h$,
where $f_h \in V_h^*$, and where the operator ${\mathcal{A}_h} : V_h \to V_h^*$ is defined by
${\mathcal{A}_h} u ( v ) = a(u,v), $ for all $u,v \in V_h$.
Hence, ${\mathcal{A}_h}$ is symmetric in the sense that for all $u,v \in V_h$,
$\< {\mathcal{A}_h}u,v \> = \< {\mathcal{A}_h}v,u \>,$
where $\< \cdot, \cdot \>$ is the duality pairing between $V_h^*$ and $V_h$.
To turn the discrete system \eqref{weak-h-1} into  a system of linear equations, written in a 
matrix/vector form, we need to introduce a basis $\{\phi_j\}_{j=1}^n$
for the  space $V_h$.
This means that any element
$v \in V_h$ can be written uniquely on the form
$v = \sum_j c_j \phi_j$. We denote the map from $\R^n$ to $V_h$ given by $ c \mapsto v$
for $\tau_h$. In a corresponding manner we define 
$\mu_h : V_h^* \to \R^n$ by $(\mu_hf)_i = \< f ,\phi_i \>$. We note that if $f \in V_h^*$ and $c \in \R^n$ then
\[
\mu_h f \cdot c = \sum_{i=1}^n \< f , \phi_i\>  c_i = \< f, \tau_h(c)\>,
\]
where $\R^n$ is equipped with the standard Euclidean inner product, and where we adopt the standard ``dot notation" for this inner product.
Hence, 
$\mu_{h} : V_h^*\to \R^n$ can be identified as  $\tau_h^*$. If $c$ is the unknown vector, $c= \tau_h^{-1}u_h$,
then the system \eqref{weak-h-1} is equivalent to the linear system 
\begin{equation}\label{lin-system-1}
{{\A}_h} c = \mu_h(f_h) \equiv \tau_h^*(f_h), 
\end{equation}
where ${{\A}_h}$ corresponds to the $n \times n$ matrix representing the operator $\tau_h^* {\mathcal{A}_h} \tau_h : \R^n \to \R^n$. The matrix ${{\A}_h}$ is usually referred to as the stiffness matrix,
and the element $({{\A}_h})_{i,j}$ is given as $a(\phi_i, \phi_j)$. Furthermore, we note that the diagram 
\begin{equation}\label{A-diagram}
\begin{diagram}
\mathbb{R}^{n} & \rTo^{\mathbb{A}_{h}} &\mathbb{R}^{n}\\
       \dTo^{\tau_{h}}       &  & \uTo^{\tau_h^*}\\
 V_{h} & \rTo^{\mathcal{A}_{h}} &V_{h}^{\ast}
\end{diagram}
\end{equation}
commutes.
However, there is a striking difference between the operator ${\mathcal{A}_h} : V_h \to V_h^*$ and its matrix representation ${{\A}_h}$. The stiffness matrix ${{\A}_h}$ depends strongly on the choice of basis,
while the operator ${\mathcal{A}_h}$ only depends on the bilinear form $a$ and the 
space $V_h$. 


For piecewise polynomial spaces of high order the choice of basis can have dramatic effect on the 
conditioning of the stiffness matrix $\A_h$. Therefore, 
there are a number of contributions in the literature discussing how to choose proper bases for 
$C^0$ piecewise polynomial spaces of high order. The purpose of these constructions is usually motivated  by the desire to control specific condition numbers
or to control the sparsity of the  resulting matrices. Let ${\mathcal{I}_h} : V_h^* \to V_h$ be the Riesz map 
given by 
\[
\< {\mathcal{I}_h} f,v \>_{L^2} = \< f,v \>, \quad f \in V_h^*,v \in V_h.
\]
The operator ${\mathcal{I}_h}{\mathcal{A}_h} : V_h \to V_h$ is symmetric and positive definite in the $L^2$
inner product. This operator is also basis independent, and its eigenvalues are given by the generalized eigenvalue problem 
\[
a(u,v) = \lambda \< u,v \>_{L^2}, \quad u,v \in V_h.
\]
If $r$ is the polynomial degree then its spectral condition number, $\kappa({\mathcal{I}_h}{\mathcal{A}_h})$,
generally may grow  like $r^4$, cf. \cite{bernardi1997spectral, hu1998bounds, schwab1998p}. 
Therefore, one possiblity is to design a special basis such that the spectral condition number 
of the stiffness matrix is much smaller than that of the operator ${\mathcal{I}_h}{\mathcal{A}_h}$, i.e.,
\[
\kappa({{\A}_h}) = \kappa(\tau_h^*{\mathcal{A}_h} \tau_h) \ll \kappa({\mathcal{I}_h}{\mathcal{A}_h}).
\]
The obvious constructions which will lead to this is to consider bases which are close to orthonormal with respect to the bilinear form $a$. 
This approach is taken  by  
Schwab in \cite{schwab1998p}, where integrated Legendre polynomials are used
to construct a basis in one space dimension.
However, the generalization of this approach to higher dimensions is not obvious.
Babu\v{s}ka and Szab\'{o} \cite{szabo1991finite} used these basis functions in a tensor product 
setting to obtain bases for 
cubical meshes in higher dimensions.
In \cite{hu1998bounds} it is established that for such bases one has an estimate for the
condition number of the stiffness matrix of the form
$\kappa(\A_{h})\lesssim r^{4(d-1)}$, while  
$\kappa(\M_{h})\lesssim r^{4d}$, where $\M_h$ is the mass matrix given by 
$(\M_h)_{i,j} = \< \phi_i, \phi_j \>_{L^2}$ and 
$d$ is the space dimension.    In particular, these estimates indicate that in  higher dimensions 
it may occur that $\kappa(\A_h)$ is much larger than its basis independent counterpart, $\kappa({\mathcal{I}_h}{\mathcal{A}_h})$.
 Similar constructions on triangular and tetrahedral meshes are based on 
orthogonal polynomials with respect to triangles and tetrahedrons constructed by the 
Duffy transform, i.e., mappings between  simplexes and cubes, cf. \cite{beuchler2012sparsity,fuhrer2015optimal, karniadakis2013spectral,
li2010optimal, zaglmayr2006high}.
In a slightly different direction,
Xin and Cai  \cite{xin2012well}  used multivariate orthogonal polynomials on simplexes
 to design $L^{2}$ hierarchical bases for the discontinuous Galerkin method. 

Our aim in this paper is to construct representations of $C^0$ finite element spaces with $L^2$ condition numbers 
which are bounded independently of the polynomial degree. If the spaces are represented by a basis this quantity
can also be characterized as the spectral condition number of the mass matrix. We will argue below that if
$\kappa(\M_{h})$ is well-behaved then the condition number of the stiffness matrix is basically controlled by 
the basis independent quantity $\kappa(\mathcal{I}_h \mathcal{A}_h)$, cf. \eqref{prod-cond} below.
In fact, we will not restrict the discussion below to bases, but
instead also allow representations of finite element spaces by frames, i.e., generating sets with redundancy.
This more general set up allows us to identify a construction  which will lead to representations with $L^2$ 
condition numbers  which are independent of the polynomial degree $r$. The  key tools we will use 
for this construction include the properties of the bubble transform, cf. \cite{falk2013bubble}.
By combining proper results for the bubble transform with general results for frames based on space decompositions,
the construction of frames with $L^2$ condition numbers bounded independently of the polynomial degree is 
reduced to a pure polynomial problem. More precisely, we need to construct Jacobi polynomials on standard simplexes
in higher dimensions, and these constructions are well known, cf. \cite{dunkl2014orthogonal}.

The rest of this paper will be organised as follows. In Section \ref{sec:notation}, we will present some notation and preliminaries needed for our discussions. 
In particular, we present a simple  bound that relates the spectral condition number of the stiffness matrix and the mass matrix,
and we present some elementary numerical examples to illustrate the sharpness of this bound.
Furthermore, we will recall 
the construction of the bubble transform and its properties.
Section \ref{sec:general-estimates} is devoted to a discussion of frames obtained by a space decomposition, where we 
focus on general estimates for the appropriate condition numbers.
The construction of the specific frames are completed in 
Section \ref{sec:H1basis}, where we explain how to
utilize well-conditioned bases on local subdomains.
We focus on the  construction of $L^2$ orthogonal local bases in Section \ref{sec:H1basis} 
{and present explicit forms of the local bases based on Jacobi polynomials on simplexes. We present numerical experiments based on these local constructions  which give results 
that are consistent with the theory. In particular, the results of the experiments 
indicate that the behavior of the frame condition numbers are robust with respect 
to perturbations of the mesh, even when the mesh is nearly degenerate.} 
Sections \ref{precond-it} is devoted to the discussion of preconditioned Krylov space methods 
for the associated frame systems, which in general will be singular.
In particular, 
we make the observation that under the assumption of standard representations of the discrete elliptic operator and the preconditioner, the conditioning of the preconditioned system is, 
in a proper sense, independent of the choice of basis or frame.
On the other hand, the representation will substantially effect the individual conditioning of the stiffness matrix and the matrix representation of the preconditioner.
However, we will argue that as long as the $L^2$ condition number of the representation stays bounded, then these matrices will roughly behave like their basis independent counterparts.
Finally, some 
concluding remarks are given in \S \ref{sec:concluding}.


\section{Notation and preliminaries}\label{sec:notation}


We assume that $\Omega$ is a bounded polyhedral domain in $\mathbb{R}^{d}$.
We recall the definition of the Sobolev spaces 
   $$
   H^{1}(\Omega):=\{u\in L^{2}(\Omega), \mathrm{grad} u\in L^{2}(\Omega)^{d}\},
   $$ 
   and the corresponding subspace of functions with vanishing trace:
   $$
   \0 H^{1}(\Omega):=\{u\in H^{1}(\Omega): \mathrm{tr}_{\partial \Omega} u=0\}.
   $$
We assume that $\mathcal{T}_{h}$ is a simplicial mesh on  $\Omega$.
If $S$ is a subset of $\R^d$ we let $\mathcal{P}_{r}(S)$ denote the polynomials with degrees less than or equal to $r$ on $S$, while the corresponding space of $C^0$-piecewise polynomials with respect to the mesh $\mathcal{T}_h$ is denoted ${\mathcal{P}_r}(\mathcal{T}_{h})$, i.e., 
$$
{\mathcal{P}_r}(\mathcal{T}_{h}):=\{u \in C^{0}(\Omega): u|_{T} \in {\mathcal{P}_r}(T), ~\forall T\in \mathcal{T}_{h} \}.
$$

\subsection{Representation of discrete operators}\label{disc-oper}
Consider a finite element system of the form  \eqref{lin-system-1}, where the space 
$V_h = {\mathcal{P}_r}(\mathcal{T}_{h})$ for a suitable $r \ge 1$, and let $\{\phi_j \}_{j=1}^n$ be any basis for $V_h$. We recall that the $n \times n$ stiffness matrix $\A_h$ is given as $\tau_h^* \mathcal A_h \tau_{h}$,
where $\mathcal A_h : V_h \to V_h^*$ is the discrete elliptic operator defined by the variational problem
\eqref{weak-h-1}. The corresponding mass matrix is the $n \times n$ matrix with elements $\< \phi_i,\phi_j \>_{L^2}$. Alternatively, $\M_h = \tau_h^* \mathcal I_h^{-1} \tau_h$, where we recall that $\mathcal I_h$ is the Riesz map, mapping $V_h^*$ to $V_h$. The matrices $\A_h$ and $\M_h$ are related by the relation 
\begin{equation}\label{A-rep}
\A_h = \tau_h^* \mathcal I_h^{-1}\mathcal I_h \mathcal A_h \tau_h = \M_h (\tau_h^{-1} \mathcal I_h \mathcal A_h \tau_h).
\end{equation}
We note that the operator $\tau_h^{-1} \mathcal I_h \mathcal A_h \tau_h$ is similar to the basis independent operator $\mathcal I_h \mathcal A_h$. A direct consequence of the identity \eqref{A-rep},
using the characterization of  the extreme eigenvalues by the Raleigh quotient, is the inequality
\begin{equation}\label{prod-cond}
\kappa({{\A}_h}) \le \kappa({\mathcal{I}_h} {\mathcal{A}_h}) \kappa(\M_h).
\end{equation}
where the spectral condition number $\kappa({\mathcal{I}_h} {\mathcal{A}_h})$ is basis 
independent, while 
$\kappa(\M_h)$ is independent of the underlying elliptic operator. In fact, \eqref{prod-cond}
follows from the stronger properties
\[
\lambda_{\max}({{\A}_h}) \le \lambda_{\max}({\mathcal{I}_h} {\mathcal{A}_h}) \lambda_{\max}(\M_h), \quad \text{and } \lambda_{\min}({{\A}_h}) \ge \lambda_{\min}({\mathcal{I}_h} {\mathcal{A}_h}) \lambda_{\min}(\M_h),
\]
where $\lambda_{\min}$ and $\lambda_{\max}$ denote the extreme eigenvalues. To see this
just observe that 
\begin{align*}
\lambda_{\max}({{\A}_h}) &= \sup_{0 \neq c\in \R^n} \frac{a(\tau_hc,\tau_hc)}{c \cdot c}
\le \sup_{0 \neq v \in V_h} \frac{a(v,v)}{\< v,v\>_{L^2} }
\sup_{0 \neq c\in \R^n} \frac{\< \tau_hc, \tau_h c\>_{L^2}}{c \cdot c}\\
&\le \lambda_{\max}({\mathcal{I}_h} {\mathcal{A}_h}) \lambda_{\max}(\M_h),
\end{align*}
and a similar argument establishes the corresponding inequality for $\lambda_{\min}$.

We can therefore conclude that if the basis is chosen such that $\kappa (\M_h)$ is properly bounded, 
then the conditioning of the 
stiffness matrix  $\A_h$ 
is no worse 
than its basis independent analog. 
To illustrate  the effect on the conditioning of the stiffness matrix $\A_h$, by controlling the $L^2$ condition number 
of the basis, we present some simple numerical examples in one space dimension. In other words, we are testing 
the sharpness of the bound \eqref{prod-cond} in the simplest possible setting.

\smallskip
\noindent
{\it Example.}
We consider the Laplace problem in one space dimension, with homogeneous Dirichlet  boundary conditions, i.e., the bilinear form $a$ is given by 
\[
a(u,v) = \int_{-1}^1 u'(x) v'(x) \, dx,
\]
and we use a mesh consisting of one interval. Therefore, the corresponding spaces $V_h$
will be given as $V_h = \0{\mathcal{P}}_{r}(\Omega)$, where $\Omega$ is the interval 
$(-1,1)$ { and $\0{\mathcal{P}}_{r}(\Omega)$ is the space of polynomials of degree on $\Omega$ which vanish on the boundary}. We will investigate the effect of choosing three different bases 
$\{\phi_{1}, \phi_{2}, \cdots, \phi_{r-1}\}$ 
for the spaces $V_h$,
by computing the condition numbers of the mass matrix $\mathbb{M}_{h}$, the stiffmess matrix $\mathbb{A}_{h}$ and  the condition number of the 
basis independent operator $\mathcal{I}_{h}\mathcal{A}_{h}$.
In fact, for any basis the latter is equal to $\kappa(\mathbb{M}_{h}^{-1}\mathbb{A}_{h})$.



Our first test is based on an $L^{2}$ orthonormal basis.  We consider the polynomials 
$(1-x)(1+x)J^{2, 2}_{r}(x), ~r=0, 1, \cdots$, where $J_{r}^{2, 2}(x)$ is the orthonormal Jacobi polynomials  on $[-1, 1]$ with respect to the weight $(1-x)^2(1+x)^2$, i.e.,
$$
\int_{-1}^{1}J^{2, 2}_{s}(x)\cdot J^{2, 2}_{t}(x) (1-x)^{2}(1+x)^{2}dx=\delta_{st},
$$
cf. Appendix A below. 
Since these polynomials form an orthonormal basis for $V_h$ $\kappa(\M_h) = 1$
for all $r$ and $\kappa(\mathbb{M}_{h}^{-1}\mathbb{A}_{h}) = \kappa(\A_h)$.
The logarithms of the latter, for increasing values of $r$,
 are shown in Figure 
\ref{fig:1DinteriorLegendre}, while $\log(\kappa(\mathbb{A}_{h}))$ are compared
to $\log(r)$ in   Figure \ref{fig:interiorlogr}.  The growth depicted here is consistent with the asympotic upper bound, $\kappa(\mathcal{I}_{h}\mathcal{A}_{h}) \lesssim r^{4}$, {cf., \cite{bernardi1997spectral}}.

\begin{figure}
\centering 
{\setlength{\abovecaptionskip}{-30pt}
\begin{minipage}{.5\textwidth}
\includegraphics[width=1.08\linewidth]{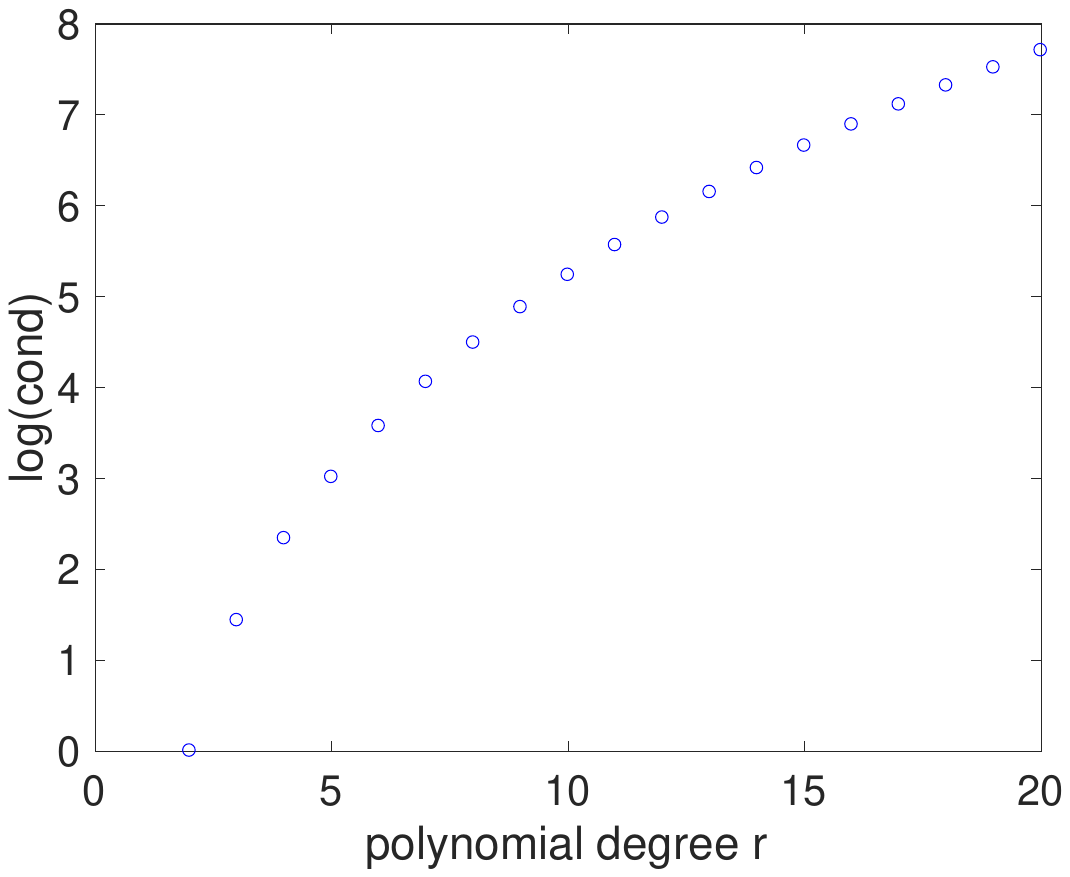}
\caption{The Jacobi basis:  $\log(\kappa(\mathcal{I}_{h}\mathcal{A}_{h})) = \log(\kappa(\mathbb{A}_{h}))$ as functions of $r$.}
\label{fig:1DinteriorLegendre}
\end{minipage}}%
\begin{minipage}{.5\textwidth}{\setlength{\abovecaptionskip}{+2pt}
\includegraphics[width=1.06\linewidth]{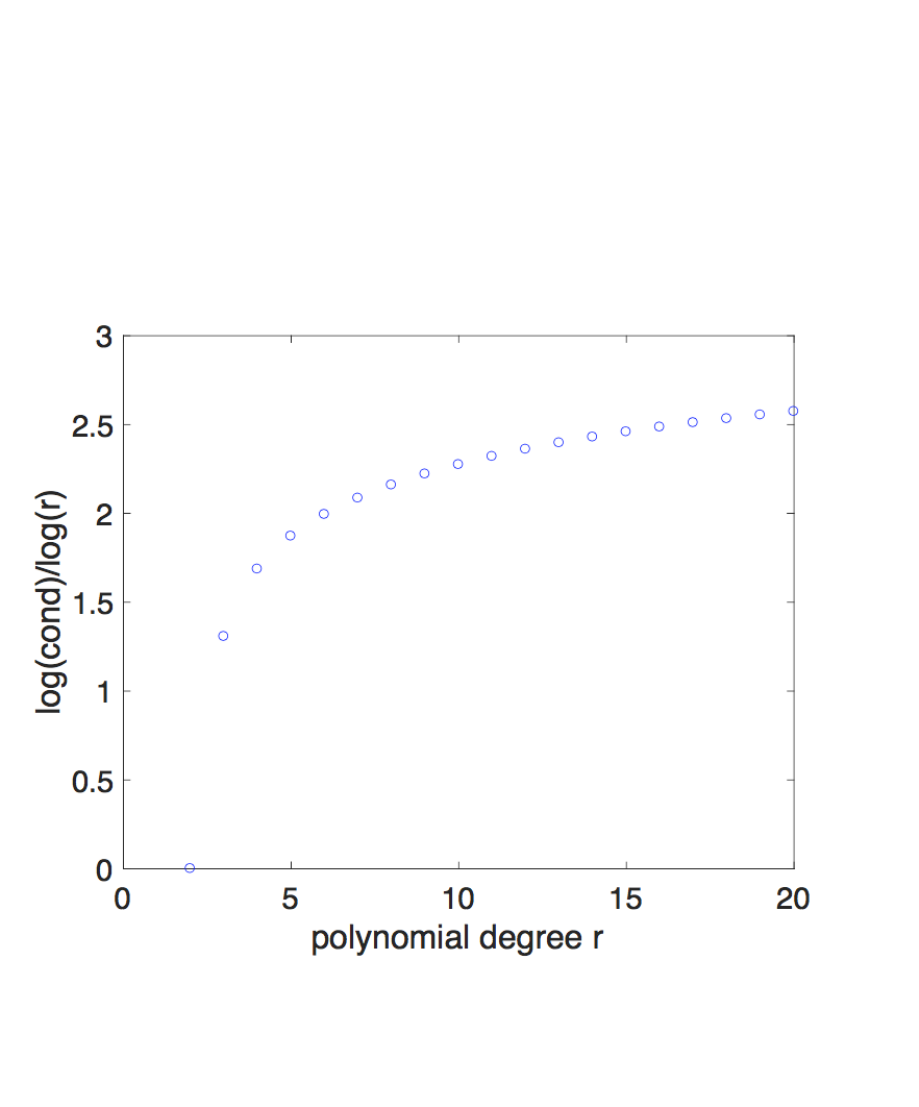}\setlength{\abovecaptionskip}{-5pt}
\caption{$\log(\kappa(\mathcal{I}_{h}\mathcal{A}_{h}))/\log(r)$ as functions of $r$.}
\label{fig:interiorlogr}}
\end{minipage}
\end{figure}

Next we consider the Bernstein basis. More precisely, for any $r \ge 2$ consider the functions 
$b_{s, r}\left (\frac{x+1}{2}\right )$, where 
$
b_{s, r}(x)={r \choose s}x^{s}(1-x)^{r-s}, \quad 1 \leq s \leq r-1.
$
 The condition numbers of the mass and stiffness matrices are shown {in  Figure \ref{fig:interior-bernstein}}, and the results indicate that both 
 $\kappa(\M_h)$ and $\kappa(\A_h)$ grow exponentially with $r$.
 This is consistent with the explicit formula,
 $\kappa(M_h) = \sqrt{2r+1 \choose r}$, which holds in the case of no boundary condidtions
\cite{ciesielskii1985degenerate, lyche2000p}).
Furthermore, we observe that $\kappa(\A_h)$ is several magnitudes larger
than the basis independent quantity $\kappa(\mathcal{I}_{h}\mathcal{A}_{h}).$

\begin{figure}[H]\centering
{\setlength{\abovecaptionskip}{-40pt}
\begin{minipage}{.5\textwidth}
\includegraphics[width=1.08\linewidth]{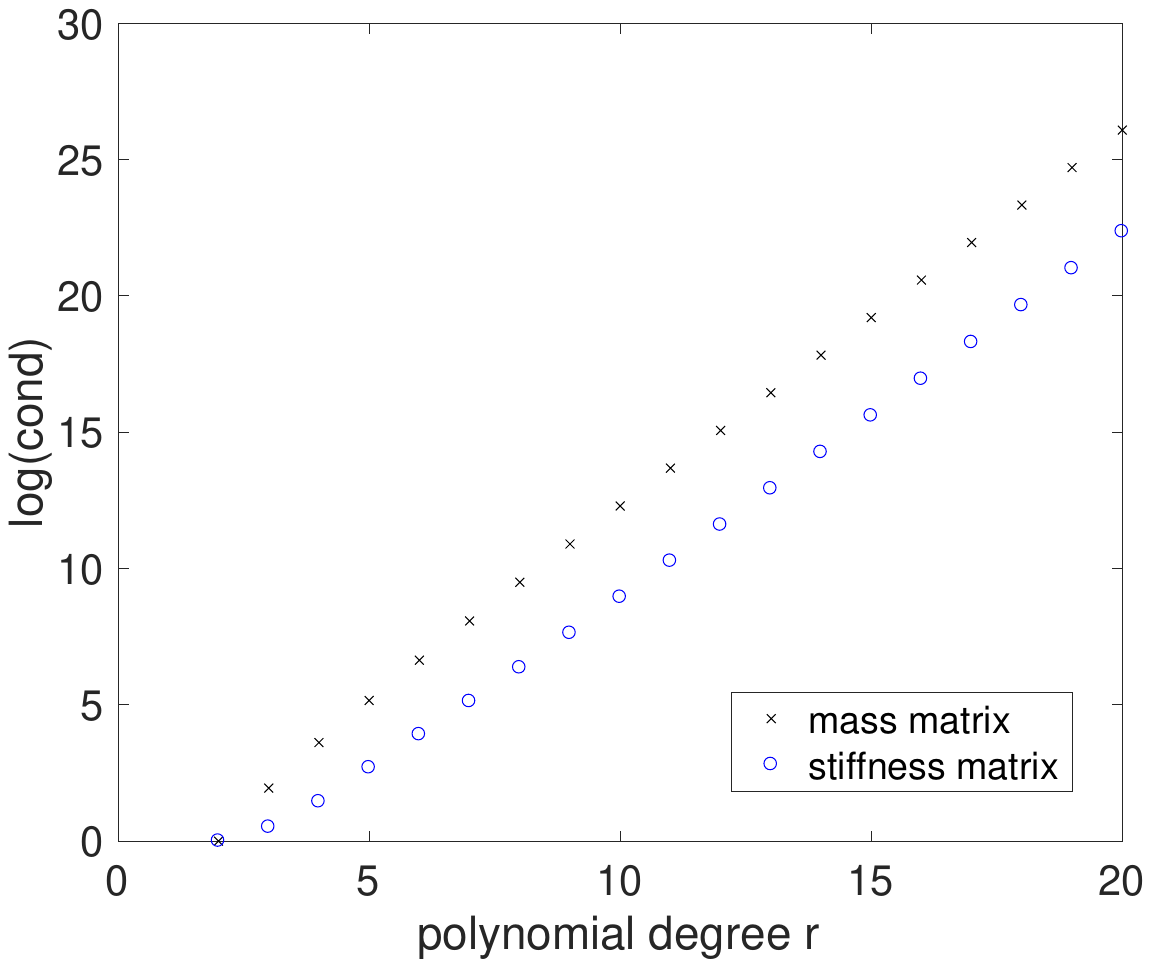}
\caption{The Bernstein basis,  $\log(\kappa(\mathbb{M}_{h}))$ and $\log(\kappa(\mathbb{A}_{h}))$ as functions of $r$.}
\label{fig:interior-bernstein}
\end{minipage}}%
\begin{minipage}{.5\textwidth}\centering{\setlength{\abovecaptionskip}{-28pt} 
\includegraphics[width=1.08\linewidth]{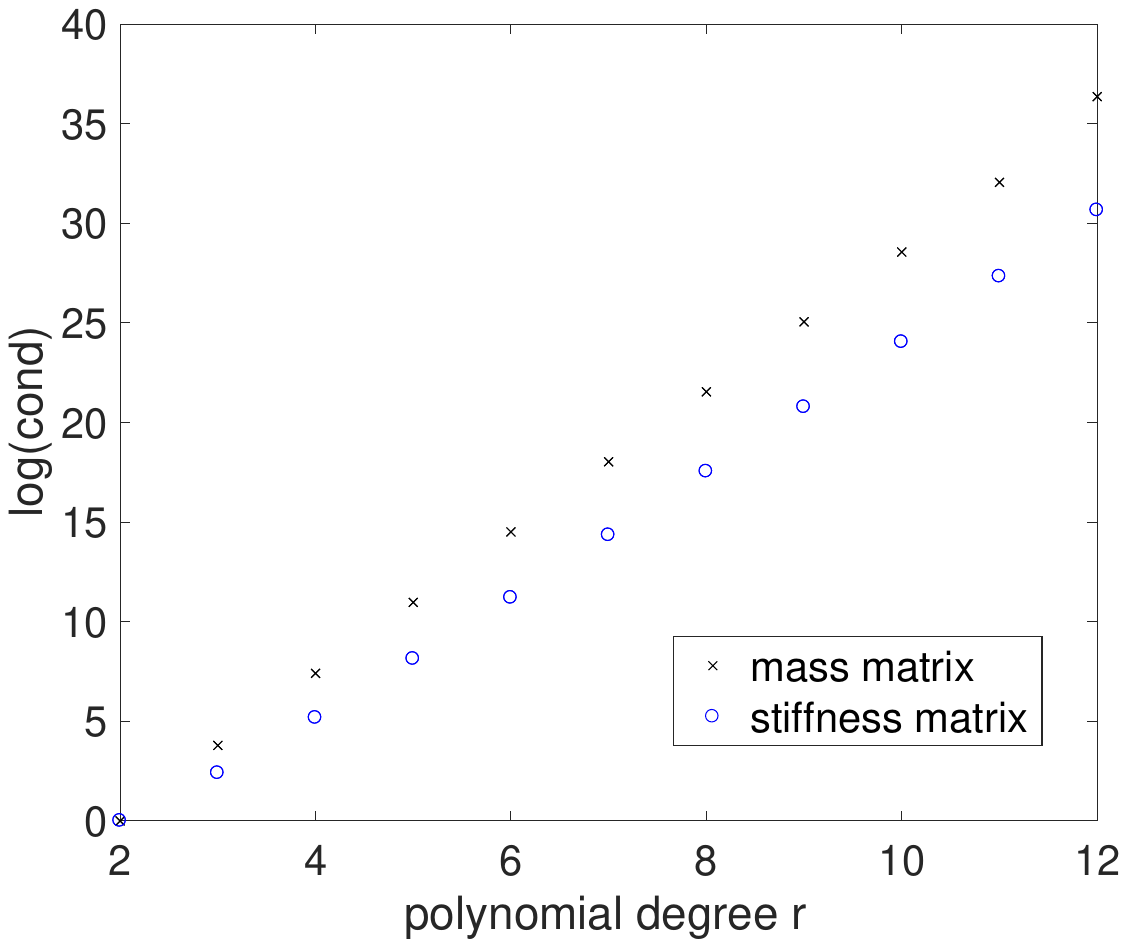}
\caption{ The power basis, $\log(\kappa(\mathbb{M}_{h}))$ and $\log(\kappa(\mathbb{A}_{h}))$ as functions of $r$.}
\label{fig:power}}
\end{minipage}

\end{figure}

Finally, we consider 
the corresponding power basis $(1+x)(1-x)x^{r-2}, ~~r=2, 3, \cdots $. The condition numbers of the mass and stiffness matrices
are in this case even much larger than for the Bernstein basis,  cf. Figure \ref{fig:power}. Due to the extremely bad condition numbers, the computations are only reliable for small values of $r$. 

The experiments just presented 
illustrate an effect of the bound \eqref{prod-cond}. 
For the Jacobi basis the 
condition number 
of the mass matrix is well controlled, and as a consequence the growth
of $\kappa(\A_h)$ is moderate.
On the other hand, for the two other bases $\kappa(\M_h)$ and 
$\kappa(\A_h)$ both grows much faster than $\kappa(\mathcal{I}_{h}\mathcal{A}_{h}).$

\subsection{The local spaces $Q_{f,r}^*$}
The rest of this paper will mostly be devoted to the construction of representations for the spaces 
${\mathcal{P}_r}(\mathcal{T}_{h})$ which admit $L^2$ condition numbers which are independent of the polynomial degree $r$. To obtain such a result we will not restrict the discussion to representations by bases, but 
we will  allow more general representations by frames. In particular, our construction will rely on  
results for the bubble transform derived  in \cite{falk2013bubble}.
To present these results, and to explain how they will be used here, we will first introduce some additional notation. 

If $\mathcal T_h$ is simplicial triangulation of $\Omega$, we let 
$\Delta_{m} = \Delta_{m}(\mathcal T_h) $ be the set of all the  subsimplexes of $\mathcal{T}_h$ of dimension $m$, while $\Delta=\cup_{m=0}^{d} \Delta_{m}$ contains all the subsimplexes.  If $T \in \mathcal T_h$ we let 
$\Delta(T)$ be the set of all subsimplexes of $T$. 
For $f\in \Delta_{m}$, the local patch, or macroelement,  $\Omega_{f}=\cup \{T\in \Delta\, :\, f \in \Delta(T) \}$ is the union of all the elements of the mesh which contains $f$.  Furthermore, $\mathcal{T}_{f, h}$ is the partition $\mathcal{T}_{h}$ restricted to $\Omega_{f}$.

 When $x_{j}\in \Delta_{0}$ is a vertex, we use $\lambda_{j}(x)$ to denote the piecewise linear function which equals one at $x_{j}$, and equals zero at other vertices. From another point of view, $\lambda_{j}$ is the  barycentric  coordinate associated to $x_{j}$, extended by zero outside the macroelement $\Omega_{x_{j}}$.  For  $m <d$ and $f=[x_{0}, x_{1}, \dots, x_{m}]\in \Delta_{m}$, i.e. the convex hull of vertices $x_{0}, x_{1}, \cdots, x_{m}$, we use $\lambda_{f}$ to denote the  vector field $(\lambda_0, \lambda_1, \ldots \lambda_m)$.
Following the approach taken in \cite{falk2013bubble} we will consider $\lambda_{f}$ as a mapping from the domain $\Omega$ to the standard simplex  $S_m^c$  in $\R^{m+1}$ given by 
   $$
   S_{m}^{c}:=\{\lambda=(\lambda_{0}, \cdots, \lambda_{m})\in \mathbb{R}^{m+1}: \sum_{j=0}^{m}\lambda_{j}\leq 1, \lambda_{j}\geq 0, ~\forall~ 0\leq j\leq m\}.
   $$   
We also define $S_m$ to be the face of $S_m^c$ opposite the origin, i.e, 
    $$
    S_{m}:=\{\lambda \in S_m^c : \sum_{j=0}^{m}\lambda_{j}= 1 \},
    $$
    such that $S_{m}^{c}=\left [0, S_{m} \right]$, i.e., $S_m^c$ is the set of all convex combinations of the origin and elements of $S_m$. The mapping $\lambda_{f}$, restricted to $\Omega_f$, is surjective but not injective (see Figure \ref{fig:map} for the case $m=1$). 
    
 If $f = [x_0,x_1, \ldots ,x_m] \in \Delta_m$ then the associated macroelement $\Omega_f$ can be characterized
 as $\Omega_f = \cap_{j=0}^m \Omega_{x_j}$, while the corresponding extended macroelement, $\Omega_f^{e} \supset
 \Omega_f$,  is defined by 
 \[
 \Omega_f^{e} = \bigcup_{j=0}^m \Omega_{x_j}.
 \]
 The pull back of the extended barycentric coordinates, $\lambda_{f}^{\ast}$, given by 
 \[
 \lambda_f^{*} v(x)  =  v(\lambda_f(x)),
 \]
maps functions on $S_{m}^{c}$ to functions on $\Omega$ which are constant and equal to $v(0)$ outside $\Omega_f^{e}$. Furthermore, if $\tr_{\partial S_m^c \setminus S_m} v = 0$ then $\lambda_f^{*}v$ vanishes on the boundary of $\Omega_f$.

The  space of polynomials of degree less than equal to $r$ which vanish on $\partial S_m^c \setminus S_m$
will be denoted $Q_{r}\left (S_{m}^{c}\right )$, i.e., 
\[
Q_{r}\left (S_{m}^{c}\right ):=\left \{v\in \mathcal{P}_{r}\left (S_{m}^{c}\right ): \tr_{\partial S_{m}^{c}\backslash  S_{m}}v=0\right  \}.
\]
By applying the pullback, $\lambda_f^*$,  to this polynomial space we obtain 
\begin{align}\label{Q-fr}
{Q_{f, r}^{\ast}}:=\lambda_{f}^{\ast}\left (Q_{r}\left ( S_{\dim f}^{c}\right ) \right ), \quad m <d.
\end{align}
The elements of this space are polynomials in the variables $\lambda_0(x), \lambda_1(x), \ldots , \lambda_m(x)$,
and they vanish on the boundary of $\Omega_f$. In other words, the space ${Q_{f, r}^{\ast}}$
can be identified with a subspace of $\0 {\mathcal{P}_r}(\mathcal{T}_{f,h})$,  the subspace of 
${\mathcal{P}_r}(\mathcal{T}_{f,h})$ which consists of functions which vanish on $\partial \Omega_f$.
In fact, in the special case when $m=d$, i.e., when  $f \in \Delta_d = \mathcal{T}_{h}$ we define the ${Q_{f, r}^{\ast}}$ to be equal to 
$\0 {\mathcal{P}_r}(f)$. 
{Alternatively, if we define $Q_{r}\left (S_{m}^{c}\right )$
to be $\0 {\mathcal{P}_r}(S_m)$ for $m=d$, then the identification \eqref{Q-fr} also holds in this case.}
The local spaces ${Q_{f, r}^{\ast}}$ 
will act as key building blocks in our construction below.

\begin{figure}[htb]
\setlength{\unitlength}{0.35cm}
\centering
\begin{picture}(30,15)
\linethickness{2pt}
\put(5,0){\line(0,1){10}}
\linethickness{1pt}
\put(5,0){\line(2,1){10}}
\put(5,10){\line(2,-1){10}}
\put(5,0){\line(-3,4){5.15}}
\put(5,10){\line(-3,-2){5.15}}
\put(20,0){\vector(0,1){11}}
\put(20,0){\vector(1,0){11}}
\put(31.5,0){$\lambda_0$}
\put(19.8,11.5){$\lambda_1$}
\linethickness{2pt}
\qbezier(20,10)(25,5)(30,0)
\linethickness{1pt}
\put(7,3.5){$x$}
\put(5.1,5){$f$}
\put(4.8,-.6){$x_0$}
\put(4.8,10.5){$x_1$}
\put(7.5,10){$\Omega_f$}
\put(23,3.5){$\lambda_f(x)$}
\put(22.5,8){$S_1$}
\put(28,10){$S_1^c$}
\qbezier(7,3)(15,9)(23,3)
\put(22.4,3.7){\vector(1,-1){1}}
\end{picture}
\caption{$\lambda_{f}: x \mapsto \lambda_f(x)$ for $f = [x_0,x_1]$ and $d=2$.}
\label{fig:map}
\end{figure}

\subsection{The bubble transform}
The bubble transform 
is a map that depends on the mesh $\mathcal T_h$, but no piecewise polynomial space occurs in the definition. In particular, it does not depend on a degree parameter $r$. In \cite{falk2013bubble}
the construction of the bubble transform was 
partly motivated by the desire to design local projections onto the piecewise polynomial spaces
$\mathcal P_r(\mathcal T_h)$ with proper bounds independent of $r$. In this paper, we will utilize 
the properties of the bubble transform to construct frames for the spaces $\mathcal P_r(\mathcal T_h)$ that admit 
$L^2$ condition numbers which are independent of  the degree parameter $r$.

The bubble transform is a map $\mathfrak B = \mathfrak B_{\mathcal T_h}$ of the form 
$$
\mathfrak{B}: H^{1}(\Omega)\mapsto \prod_{f\in\Delta}\0 H^{1}(\Omega_{f}).
$$
It is a tool to decompose an $H^{1}$ function defined on $\Omega$ into components $B_f u$ with local support in $\Omega_f$. More precisely,
$$
u=\sum_{f\in \Delta}B_{f}u, \quad \forall u\in H^{1}(\Omega),
$$
where
\[
B_{f}: H^{1}(\Omega_f^{e})\mapsto \0 H^{1}(\Omega_{f})
\]
 gives the component of $u$ which is supported on $\Omega_{f}$. In particular, we observe that 
 the operator $B_f$ is local in the sense that $B_fu$ depends on $u|_{\Omega_f^{e}}$.
 Another key property of the map $\mathcal B$ is that it is invariant with respect to the piecewise polynomial spaces $\mathcal P_r(\mathcal T_h)$, i.e., if $u \in \mathcal P_r(\mathcal T_h)$
 then $B_fu \in \0 \mathcal P_r(\mathcal T_{f,h})$.

The bubble transform has a recursive definition.  We briefly recall its construction, but for more details we refer to \cite{falk2013bubble}. For 
$m=0, 1, \dots ,d-1$, $B_fu$ is of the form
\[
B_{f}u:=   \lambda_{f}^{\ast}\circ K_{m}\circ A_{f} u^{m}, \quad\forall f\in \Delta_{m},
\]
where 
\begin{align}\label{def:um}
u^{m}:=\left ( u-\sum_{g\in \Delta_{j}, j<m} B_{g}u\right ),
\end{align}
while $B_T u = u^d|_T$ if $T \in \Delta_d = \mathcal T_h$.
The pull back $\lambda_f^*$, mapping functions on 
$S_m^c$ to functions on $\Omega$, is discussed above.
The operator $A_f$ is an average operator, while $K_m$ is refereed to as a cut-off operator. If $f = [x_0, x_1, \ldots ,x_m]$ then  for any $\lambda \in S_m^c$
\[{
A_fu(\lambda) = |\Omega_{f}|^{-1} \int_{\Omega_f} u(G_m(\lambda,y)) \, dy, }
\]
where $G_m : S_m^c \times \Omega_f \to \Omega_f$ is given by 
\[{
G_m(\lambda,y) = y + \sum_{j=0}^m \lambda_j(x_j - y).}
\]
The operator $A_f$ maps a function $u$ defined on $\Omega_f$ to a function $A_f u$ defined on $S_m^c$, and it is a smoothing operator in the sense that for any $u \in L^2(\Omega)$ the function $A_f u$ will have point  values away from the simplex $S_m$. 
On the other hand,  $\tr_{S_m} A_fu$ corresponds exactly to $\tr_f u$.
Furthermore, the operator $A_f$  has the property
that it maps piecewise polynomials into polynomials. More precisely, if 
$u \in \mathcal P_r(\mathcal T_{f,h})$ then $A_fu \in \mathcal P_r(S_m^c)$.

The cut-off operator $K_m$ maps the set of functions defined on $S_m^c$ to itself.
Its key property is that it preserves the trace on $S_m$, i.e., $\tr_{S_m} K_m v = \tr_{S_m}  v$,
while it kills the trace on the rest of the boundary of $\partial S_m^c$, 
i.e, $\tr_{\partial S_m^c \setminus S_m} K_m v = 0$. In addition, $K_m$ is polynomial preserving in the sense that if $v \in \mathcal P_r(S_m^c)$, with $\tr_{\partial S_m} v = 0$, then $K_m v \in \mathcal P_r(S_m^c)$.

The key properties of the bubble transform are stated in \cite[Section 4]{falk2013bubble}.
For the convenience of the readers, and since most of these properties  are essential for the discussion below, 
we summerize the main results here.
\begin{enumerate}
\item
The construction using barycentric coordinates: For $f \in \Delta_m$ and  $0 \le m <d$
$$
B_{f}u(x)=(\lambda_{f}^{\ast}\circ K_{m}\circ A_{f})u^m(x)= (K_{m}\circ A_{f})u^m(\lambda_{f}(x)),
$$
where $u^m$ is defined by \eqref{def:um}, while $B_fu = u^d|_f$ if $f \in \Delta_d$.
\item
The boundedness in $L^2$ and $H^1$: 
\begin{align}\label{stability-bubble}
\sum_{f\in \Delta}\left \|B_{f}u\right \|_{L^{2}}^{2}\leq b\|u\|^{2}_{L^{2}}, 
\quad \text{and } \sum_{f\in \Delta}\left \|B_{f}u\right \|_{H^1}^{2}\leq b\|u\|^{2}_{H^1}, 
\end{align}
where $b$ is a generic constant not depending on the function $u$. 
\item
Partition of unity: If  $u\in L^{2}(\Omega)$ then 
$
\sum_{f\in \Delta}B_{f}u=u,
$
\item 
Local support:
If $u \in H^1(\Omega)$ then $B_f u \in \0 H^1(\Omega_f)$.
\item
Polynomial preserving: If $u \in {\mathcal{P}_r}(\mathcal{T}_{h})$ then 
$
B_{f}u \in {Q_{f, r}^{\ast}}.
$
\end{enumerate}
We also note that if $u = \sum_{f\in \Delta}u_{f} $, where $\mathrm{supp}(u_{f})\subset \Omega_{f}$
then 
\begin{equation}\label{overlapping-bounds}
a\| u \|_{L^2}^2\le  \sum_{f\in \Delta }\|u_{f}\|_{L^{2}}^{2},
\end{equation}
where the constant $a$ only depends on the number of overlaps of the subdomains $\Omega_f$.
Therefore, by combining this with \eqref{stability-bubble} we obtain that the norms
$\| u \|_{L^2}$ and $(\sum_{f\in \Delta }\|B_{f}u\|_{L^{2}}^{2})^{1/2}$ are equivalent.

The bubble transform suggests a decomposition of the finite element spaces ${\mathcal{P}_r}(\mathcal{T}_{h})$ of the form 
 \begin{align}\label{bubble-decomposition}
{\mathcal{P}_r}(\mathcal{T}_{h})=\sum_{f\in \Delta}{Q_{f, r}^{\ast}}.
 \end{align}

In fact, this decomposition follows directly from the properties above. The spaces $Q_{f,r}^*$ are local
spaces consisting of piecewise polynomials with support on $\Omega_f$. On the other hand,
the sum above is in general not direct. To see this, we observe that that if $y$ is a vertex, i.e., 
$y\in \Delta_{0}$ then 
\begin{align}\label{vertex-space}
Q_{y, r}^{\ast}=\mathrm{span}\left \{\lambda_{y}(x), \lambda_{y}^{2}(x), \cdots, \lambda_{y}^{r}(x)\right \},
\end{align}
where $\lambda_{y}$ is the extended barycentric coordinate of the vertex $y$.
In particular, the function $u(x) = \lambda_y(x) (1 - \lambda_y(x))$ is an element of $Q_{y, r}^{\ast}$
for $r \ge 2$. Let $x_1, x_2, \ldots x_k$ be the other vertices in $\Delta_0(\mathcal T_{f,h})$ with correponding exteded barycentric coordinates $\lambda_1, \lambda_2, \ldots , \lambda_k$.
Then the function $u$ can alternatively be expressed as 
\[
u(x) = \lambda_y(x) \sum_{j=1}^k \lambda_j (x).
\]
Furthermore, $\lambda_y \lambda_j \in Q_{f_j, r}^{\ast}$, where $f_j = [y,x_j] \in \Delta_1$.
We conclude that the function $u$ is both in $Q_{y, r}^{\ast}$ and $\sum_j Q_{f_j, r}^{\ast}$,
and therefore the sum \eqref{bubble-decomposition}
is not direct. Similar redundancies also appear for the spaces $Q_{f, r}^{\ast}$ for simplices of higher dimensions. Therefore, if we want to  utilize the decomposition \eqref{bubble-decomposition},
and bases for the local spaces $Q_{f, r}^{\ast}$, to represent 
the functions in the spaces ${\mathcal{P}_r}(\mathcal{T}_{h})$  we are forced to study representations by frames.

 Interpreting the properties of the bubble transform for the decomposition \eqref{bubble-decomposition} of ${\mathcal{P}_r}(\mathcal{T}_{h})$ the following result is obtained.
 
 \begin{theorem}\label{thm:stable-decomposition}
The decomposition \eqref{bubble-decomposition}
is stable in the sense that there exists $B_{f}:  {\mathcal{P}_r}(\mathcal{T}_{h}) \to {Q_{f, r}^{\ast}}, ~\forall f\in \Delta$, and a positive constant $b$ such that
\begin{align}\label{transform-B}
 \sum_{f\in \Delta}\|B_{f}u\|_{L^2}^{2}\leq b \|u\|_{L^2}^{2}.
\end{align}
Furthermore, as a result of the finite overlapping property of the mesh topology, there exists a positive constant $a$ such that
\begin{align}\label{a-bound-B}
a\left \| \sum_{f\in \Delta}u_{f}  \right \|_{L^2}^{2}\leq \sum_{f\in \Delta}\|u_{f}\|_{L^2}^{2}, \quad \forall  u = \{u_f\} \in {\Pi_{f\in \Delta}Q_{f, r}^{\ast}}.
\end{align}
\end{theorem}

\section{Estimates of frame condition numbers}\label{sec:general-estimates}
We will utilize the decomposition \eqref{bubble-decomposition} to obtain a well conditioned 
representation of functions in the space ${\mathcal{P}_r}(\mathcal{T}_{h})$.
More precisely, we will combine the decomposition \eqref{bubble-decomposition} with
a basis for each of the spaces $Q_{f, r}^{\ast}$. Due to the  redundancy of the decomposition  
\eqref{bubble-decomposition} this will lead to a spanning set for the functions in 
$ {\mathcal{P}_r}(\mathcal{T}_{h})$ where the elements are not linearly independent, i.e., we obtain 
representations by frames, cf.  \cite{oswald1997frames}. Therefore, we will give a 
brief review of representations by frames. In particular, we will discuss 
frames obtained from space decompositions.

Throughout this section we use $W$ to denote a real, finite dimensional Hilbert space with  inner product $\< \cdot, \cdot\>_{W}$. 
Roughly speaking, a frame is a set of generators which allow redundancy.  
In other words, if $\Phi=\{\phi_{1}, \phi_{2}, \cdots\}$  then each element $u \in W$
can be expressed as a linear combination $u = \sum_{k}c_{k}\phi_{k}$, but this representation is in general not unique.
The condition number of the frame $\Phi$,  ${\mathcal{K}}(\Phi)$,
is defined as 
$${\mathcal{K}}(\Phi)=\beta/\alpha,$$ where 
\begin{equation}\label{lower/upper}
\alpha=\inf_{u\neq 0}\frac{\|u\|_{W}^{2}}{\inf_{u=\sum_{k}c_{k}\phi_{k}}\|c\|_{l^{2}}^{2}}, \quad \beta=\sup_{u\neq 0}\frac{\|u\|_{W}^{2}}{\inf_{u=\sum_{k}c_{k}\phi_{k}}\|c\|_{l^{2}}^{2}}.
\end{equation}
In other words, $\alpha$ and $\beta$ are the optimal constants such that the bounds
\[
\alpha \inf_{c: u=\sum_{k}c_{k}\phi_{k}}\|c\|_{l^{2}}^{2}\leq \|u\|_{W}^{2}\leq \beta \inf_{c: u=\sum_{k}c_{k}\phi_{k}}\|c\|^{2}_{l^{2}}
\]
holds. Therefore, ${\mathcal{K}}(\Phi)$ is the natural concept to relate the norm of $u$
to the norm of its coefficients measured in $l^2$.

\begin{remark}\label{condition-number}
If $\Phi$ is a basis then it is well known that ${\mathcal{K}}(\Phi)$ is equal to the
spectral condition number  of the corresponding ``mass matrix," with elements $\<\phi_i, \phi_j\>_W$.
In fact, the parameters $\alpha$ and $\beta$, given in \eqref{lower/upper},
are exactly the smallest and largest eigenvalue of the mass matrix. In the case of a frame, the mass matrix will in general be singular.
In this case $\beta$ is still the largest eigenvalue of the mass matrix,  while $\alpha$ is the smallest positive eigenvalue.
 \end{remark}

In fact, a similar characterization can be given in the case of frames, cf. Section~\ref{solvebyframe} below.

\subsection{Frames based on space decomposition} 
To give a general description of  frames based on space decomposition, we 
assume that the space $W$ admits a decomposition of the form 
\begin{align}\label{decomposition}
W=\sum_{j=1}^{J}W_{j},
\end{align}
where $W_{j}, j=1,\cdots, J$ are  subspaces of $W$.
The decomposition is not assumed to be direct, but we
assume that there exists a positive constant $a$ such that for any  
$u = \{u_{j}\} \in \prod_{j=1}^J W_{j}$, we have
\begin{align}\label{a-bound}
a\left \| \sum_{j=1}^{J}u_{j}  \right \|_{W}^{2}\leq \sum_{j=1}^{J}\|u_{j}\|_{W}^{2}.
\end{align}
Furthermore, we
assume that there is a positive constant $b$ such that all $u \in W$ admits a decomposition 
$u = \sum_{j=1}^J u_j$, $u_j \in W_j$, where 
\begin{align}\label{transform}
 \sum_{j}\|u_j\|_{W}^{2}\leq b \|u\|_{W}^{2}.
\end{align}
Of course, due to \eqref{stability-bubble} and \eqref{overlapping-bounds},
these bounds  are known to hold 
for the spaces ${\mathcal{P}_r}(\mathcal{T}_{h})$ with $L^2$ inner product.
We will use the decomposition  \eqref{decomposition} to define a frame for $W$. 
More precisely, for each $j$ let $\{\phi_{j, k}\}_{k}$ be a basis for the space  $W_{j}$.
The frame $\Phi$ is then given as 
 $ \{\phi_{j, k}\}_{1\leq j\leq J, 1\leq k\leq N_{j}}$, where $N_{j}$ is the dimension of $W_{j}$. 
For each $j$ we assume that $0 < \alpha_{j} < \beta_{j}$ are the optimal constants such that 
\begin{align}\label{local}
\alpha_{j}\sum_{k}c_{j, k}^{2}\leq \left \|\sum_{k}c_{j, k}\phi_{j, k}\right \|_{W}^{2}\leq \beta_{j}\sum_{k}c_{j, k}^{2}, \quad\forall \{c_{j, k}\}_k.
\end{align}
In other words, 
$
\mathcal{K}_{j}=\beta_{j}/\alpha_{j}
$
is the condition number for the basis $\{\phi_{j, k}\}_{k}$ of $W_j$.

\begin{theorem}\label{thm:estimate1}
Assume that the decomposition \eqref{decomposition} satisfies 
\eqref{a-bound}, \eqref{transform} and \eqref{local}. The frame $\Phi = \{\phi_{j, k}\}_{j, k}$ introduced above satisfies
\[
\mathcal K(\Phi) 
\le a^{-1}b  \left (\max_{1 \le j \le J} {\mathcal K}_j\right ) \max_{1\leq j,l \leq J} \frac{\alpha_{j}}{\alpha_{l}}.
\]
\end{theorem}

\begin{proof}
Let $\alpha$ and $\beta$ be the two constants defined by \eqref{lower/upper}.
We will show that 
\begin{equation}\label{bound-alpha-beta}
\alpha \geq b^{-1} \min_{1\leq j \leq J} \alpha_{j}, \quad \text{and } 
\beta\leq a^{-1} \max_{1\leq j\leq J} \beta_{j}.
\end{equation}
From these bounds we immediately obtain 
\[
\mathcal K(\Phi) = \frac{\beta}{\alpha} \le a^{-1}b  \max_{1\leq j,l \leq J} \frac{\beta_{j}}{\alpha_{l}}
\le a^{-1}b  \left (\max_{1 \le j \le J} {\mathcal K}_j\right ) \max_{1\leq j,l \leq J} \frac{\alpha_{j}}{\alpha_{l}},
\]
which is the desired bound. Therefore, it is enough to establish the bounds given by 
\eqref{bound-alpha-beta}.

To show the first inequality, let $u$ be any element in $W$, and $u = \sum_j u_j$
a decomposition of the form \eqref{decomposition} satisfying \eqref{transform}.
Furthermore, let $c = \{c_{j,k} \}$ be the unique coefficients such that $u_j = \sum_k c_{j,k}\phi_{j,k}$. Then 
\[
\|c\|_{l^{2}}^{2} 
\leq \sum_{j=1}^J \alpha_{j}^{-1}\left\|\sum_{k=1}^{N_j}c_{j, k}\phi_{j, k}\right\|_{W}^{2}
\leq \left ( \max_{1\leq j \leq J}\alpha_{j}^{-1}\right ) \sum_{j=1}^{J}\|u_j\|_{W}^{2}\\
\leq b\left ( \max_{1\leq l\leq J}\alpha_{l}^{-1}\right )\|u\|_{W}^{2}.
\]
This implies
\[
\|u\|_{W}^{2}\geq b^{-1}\left (\min_{1\leq j \leq J}\alpha_{j}\right )\inf_{u=\sum_{j, k}c_{j,k}\phi_{j,k}}\|c\|_{l^{2}}^{2},
\]
and the first inequality of \eqref{bound-alpha-beta} follows by taking infimum over all elements of 
$u$ of $W$.

On the other hand, for any coefficient $c=\{c_{j,k}\}$, we define $u_{j}=\sum_{k}c_{j,k}\phi_{j, k}$ and  $u=\sum_{j}u_{j}$.  We then have
\[
a\|u\|_{W}^{2} \leq \sum_{j=1}^J\left\|u_{j}\right\|_{W}^{2}
\leq \sum_{j=1}^J \beta_{j}\sum_{k=1}^{N_j} c_{j,k}^{2} 
\leq \left ( \max_{1\leq j\leq J} \beta_{j}\right )\sum_{j, k}c_{j, k}^{2}.
\]
Therefore,  for any $u \in W$,  we have
$$
\|u\|_{W}^{2} \leq a^{-1}\left ( \max_{1\leq j\leq J} \beta_{j}\right )
\inf_{u=\sum_{j, k}c_{j,k}\phi_{j,k}}\|c\|_{l^{2}}^{2},
$$
and hence the second bound of \eqref{bound-alpha-beta} follows.
\end{proof}

The result above shows that the condition number of the frame $\Phi$ is bounded by the 
constants $a$ and $b$, derived from the decompostion \eqref{decomposition}
and the local condition numbers $\mathcal K_j$. In addition, the factor 
$\max_{j,k} \alpha_{j}/\alpha_{k}$, which we will refer to as a scaling factor, appears.
This factor will be small if all the local condition numbers  $\mathcal K_j$ are small,
and if the local bases $\{ \phi_{j,k} \}_k$ in addition are scaled similarly. In fact, 
the appearance of 
this factor is similar to a well known phenomenon.
Consider a block diagonal matrix of the form  
$$
\left (
\begin{array}{cc}
I_1 & 0\\
0 & \epsilon I_2
\end{array}
\right ),
$$
 where $I_1$ and $I_2$ are identity matrices of proper dimensions, and where $\epsilon>0$
 is a real parameter. Then each block 
  has condition number $1$, while the full matrix has condition number $\epsilon^{-1}$
  due to the different scaling of the two blocks.

\subsection{The bubble decomposition of ${\mathcal{P}_r}(\mathcal{T}_{h})$}\label{bubble-frame}

We end this section by applying Theorem \ref{thm:estimate1}
to the decomposition \eqref{bubble-decomposition} of the spaces 
${\mathcal{P}_r}(\mathcal{T}_{h})$. For each $f \in  \Delta$ let $N_{f}$ denote  the dimension of the space ${Q_{f, r}^{\ast}}$.  We will see below that it is possible to 
construct a basis for each of the spaces ${Q_{f, r}^{\ast}}$ such that 
all are well conditioned in $L^2$,  and with a comparable scaling. Therefore, consider the set up 
when we have a basis for each of the spaces ${Q_{f, r}^{\ast}}$ of the form
$\Phi_f= \phi_{f, 1}, \cdots, \phi_{f, N_{f}}$, satisfying
 \begin{align}\label{local-FEM}
\alpha_{0}\sum_{k}c_{f, k}^{2}\leq \left \|\sum_{k}c_{f, k}\phi_{f, k}\right \|_{L^{2}}^{2}\leq \beta_{0}\sum_{k}c_{f, k}^{2}, \quad\forall f\in \Delta,~ \{c_{f, k}\},
\end{align}
where the positive constants $\alpha_0$ and $\beta_0$ are independent of $f \in \Delta$.
By combining 
Theorem \ref{thm:estimate1} with Theorem \ref{thm:stable-decomposition} we immediately obtain the following.

 \begin{corollary}\label{cor-bubble-frame}
Let $\Phi = \{ \Phi_{f} \}_{f\in \Delta}$ be the frame representation of the space
${\mathcal{P}_r}(\mathcal{T}_{h})$ just introduced, and  satisfying \eqref{local-FEM}.
 We have the estimate:
 \begin{align}\label{ab-bound}
 {
 {\mathcal{K}}(\Phi) \leq a^{-1}b (\alpha_0^{-1}\beta_0),}
 \end{align}
 where $a$ and $b$ are the constants appearing in \eqref{stability-bubble} and 
 \eqref{overlapping-bounds}.
  \end{corollary}
 
 \begin{proof}
 It is a consequence of \eqref{local-FEM} that the condition number of each local basis,
 $\Phi_{f}$ of ${Q_{f, r}^{\ast}}$, is bounded by $\alpha_0^{-1}\beta_0$,
 and that the same bound holds for the scaling factor appearing in Theorem \ref{thm:estimate1}.
 The result therefore follows from this theorem.
 \end{proof}
 
 \begin{remark}\label{orthonormal}
 We note that in the special case  when each of the local bases  $\Phi_f = \phi_{f, 1}, \cdots, \phi_{f, N_{f}}$
 is orthonormal then the bound above reduces to 
 $\mathcal K(\Phi) \leq a^{-1}b$, i.e, the condition number of the frame is bounded entirely by the two constants given in \eqref{stability-bubble} and 
 \eqref{overlapping-bounds}.
\end{remark}
 

\section{Construction of bases for the local spaces}\label{sec:H1basis}
Based on the discussion above, cf. Corollary~\ref{cor-bubble-frame},
 we can conclude that to obtain a well-conditioned 
frame for the spaces ${\mathcal{P}_r}(\mathcal{T}_{h})$, 
it is enough to construct bases for the local spaces 
$Q_{f,r}^*$ which are uniformly well-conditioned in $L^2$. More precisely, it is enough to 
construct bases $\Phi_f$ for  the spaces $Q_{f,r}^*$ such that condition \eqref{local-FEM} holds.
In the special case when $\dim f = d$ then $Q_{f,r}^* = \0 {\mathcal{P}_r}(f)$, and the construction of a basis for this space is well known. We return to this case in the Appendix below. When $\dim f < d$
we recall from Section~\ref{sec:notation} that $Q_{f,r}^*$ is defined by a pull back, with respect to the map 
$\lambda_f : \Omega \to S_m^c$,  of the polynomial space $Q_r(S_m^c)$, where $S_m^c$ is a reference simplex in $\R^{m+1}$
and $f \in \Delta_m(\mathcal{T}_h)$. Therefore, we will construct a basis for the space $Q_{f,r}^*$ by utilizing a basis
for $Q_r(S_m^c)$.

\subsection{Construction of local bases}
Element of the space $Q_r(S_m^c)$ is of the form $\lambda_0 \cdot \lambda_1 \cdots \lambda_m \, p \equiv (\Pi\lambda)_m\, p$,
where $p \in \mathcal{P}_{r-m-1}(S_m^c)$. Therefore, any basis for the space $\mathcal{P}_{r-m-1}(S_m^c)$ leads to a corresponding basis for $Q_r(S_m^c)$. To proceed we recall the fact from \cite[formula (5.6)]{falk2013bubble},
that if $\phi$ is any smooth function on $S_m^c$ then 
\begin{equation}\label{change-of-var}
\int_{\Omega_f} (\lambda_f^* \phi)(x)  \, dx = c_f \int_{S_m^c} \phi(\lambda) b(\lambda)^{d-m-1} \, d\lambda,
\end{equation}
where $b(\lambda) = 1 - \sum_{j=0}^m \lambda_i$, and $c_f$ is a scaling factor depending on the geometry of 
the macroelement $\Omega_f$. In fact,  in the notation of \cite[Section 5]{falk2013bubble} the constant $c_f$
is given by 
\[
c_f = \int_{f^*} J(f,q) \, dq,
\]
where $f^*$ is a piecewise flat manifold of dimension $d - 1 - \dim f$ contained in $\partial \Omega_f$,
and $J(f,q)$ is a piecewise constant function on $f^*$. However, for the discussion here it suffices to observe that 
$c_f = \mathcal{O}(h_f^{d})$, uniformly with respect to a family of shape regular meshes. Here $h_f$ is a local parameter representing the diameter of $T \in \mathcal{T}_h$, i.e. $h_f$  represents the size of the elements contained in $\Omega_f$.
As a consequence, if $u$ and $v$ are orthogonal functions on $S_m^c$, with respect to the weight functions $b(\lambda)^{d-m-1}$,
then $\lambda_f^*u$ and $\lambda_f^*v$ are $L^2$ orthogonal functions on $\Omega_f$. Alternatively, if $p$ and $q$
are elements of $\mathcal{P}_{r-m-1}(S_m^c)$, which are orthogonal with respect to the weight function 
$w_m(\lambda) := (\Pi\lambda)_m^2b(\lambda)^{d-m-1}$, then the corresponding functions $\lambda_f^* [(\Pi\lambda)_mp]$ and 
$\lambda_f^* [ (\Pi\lambda)_mq]$ are $L^2$ orthogonal functions belonging to the space $Q_{f,r}^*$. Furthermore, the norm of $\lambda_f^* [(\Pi\lambda)_mp]$ in $L^2(\Omega_f)$ is equivalent to $h_f^{d/2}$ times the corresponding weighted $L^2$ norm of $p$ on $S_m^c$.
Therefore,  the problem of 
constructing $L^2$ orthogonal  and uniformly scaled bases for the local spaces $Q_{f,r}^*$, is equivalent to the construction of  
bases for the polynomial spaces $\mathcal{P}_{r-m-1}(S_m^c)$, which are orthogonal with respect 
to the weight function $w_m$, and uniformly scaled. {Actually,  since the scaling factor $c_{f}$ only depends on $f$ and is uniform for all the bases associated with $f$, any orthonormal bases on $\mathcal{P}_{r-m-1}(S_m^c)$ with the weight  $w_m$ will transform to bases of $Q_r(S_m^c)$ with condition number one.} 

To construct a polynomial basis, which is orthogonal with respect to a polynomial weight function, corresponds to 
the study of Jacobi polynomials. 
Single variate orthogonal polynomials are of course well studied,
but there are also explicit formulas for Jacobi polynomials with respect to simplexes in higher dimensions. The most popular approach to construct orthogonal polynomials in higher dimensions is to use a transform between simplexes and cubes, referred to as ``the  Duffy transform'' or  ``the Koorwinder method" \cite{koornwinder1975two}.   Furthermore, hierarchical constructions of orthogonal polynomial can be found in \cite{dunkl2014orthogonal}. 

\subsection{Numerical results}
As an example, we present explicit formulas of our frames and explore the condition numbers of the matrices by numerical computation. This will verify our theoretical results and show that the frame is well-conditioned and the constants $a$ and $b$ in \eqref{ab-bound} are bounded on various regular or irregular meshes. We will only consider 
the two dimensional case, i.e., the space dimension $d$ is equal to  $2$.

Let $J^{\alpha, \beta}_{s}(x)$ be the standard single variate Jacobi polynomial of degree $s$ with weight $(\alpha, \beta)$ defined on $[-1, 1]$ (see also Appendix below). Define 
$$
\tilde{J}_{s}^{\alpha, \beta}(\xi):=J^{\alpha, \beta}_{s}(2\xi -1),
$$
 so that $\{\tilde{J}_{s}^{\alpha, \beta}(\xi)\}, s=0, 1,  \cdots$ is a single variate orthogonal basis on $[0, 1]$ with weight $\omega^{\alpha,\beta}:=(1-\xi)^{\alpha}\xi^{\beta}$. 
The explicit formula for our frames are derived from the 
corresponding basis for the spaces $Q_{r}\left (S_{m}^{c}\right )$ for $m=0,1,2$, cf.
\eqref{Q-fr}, where we recall that 
$Q_{r}\left (S_{2}^{c}\right )  =\0 {\mathcal{P}_r}(S_2)$.
Let $\lambda_{0}, \lambda_{1}, \cdots, \lambda_{m}$ be the barycentric coordinates
of $S_m$ extended to $S_{m}^{c}$. The explicit bases for the spaces 
$Q_{r}\left (S_{m}^{c}\right )$
 can be given as follows:
 \begin{align}
 \tilde{J}_{s}^{1, 2}(\lambda_{0})\lambda_{0},\quad s\leq r-1, \label{vertex-1}\\
 \tilde{J}^{0, 2}_{s_{1}}\left (\frac{\lambda_{0}}{1-\lambda_{1}}\right )(1-\lambda_{1})^{s_{1}}\tilde{J}_{s_{0}}^{2s_{1}+3, 2}(\lambda_{1})\lambda_{0}\lambda_{1}, \quad s_{0}+s_{1}\leq r-2, \label{edge1}\\
 \tilde{J}^{2, 2}_{s_{1}}\left (\frac{\lambda_{1}}{1-\lambda_{0}}\right )(1-\lambda_{0})^{s_{1}}\tilde{J}_{s_{0}}^{2s_{1}+5, 2}(\lambda_{0})\lambda_{0}\lambda_{1}\lambda_{2},~ s_{0}+s_{1}\leq r-3,\label{interior}
 \end{align}
 for the three cases $m= 0,1,2$. In fact, these functions 
 are special cases of the simplicial orthogonal polynomials \eqref{nd-orthogonal0}, cf.  the Appendix below. 
The corresponding bases for the spaces $Q_{f, r}^{\ast}$ are given by \eqref{Q-fr}
for $\dim f = m$.
 We note that the functions in \eqref{edge1} and \eqref{interior} do not have rotational symmetry, meaning that exchanging two barycentric coordinates in the formulas will lead to different bases. In the assembling of matrices permuting the variables in  \eqref{edge1} may facilitate the match of orientation. 

To avoid considering effects from inaccurate numerical quadrature, in the computation we use Gauss type formulas on triangles with order 20  \cite{zhang2009set}. Since in the tests below we only consider polynomials of degree less than 10, the numerical quadrature will lead to exact integration.  Due to the roundoff error, the mass matrices of the frames do not have precise zero eigenvalues. Therefore we will use a tolerance $5 e{-11}$
to select nonzero eigenvalues  and thus will consider all numbers near the machine precision as zero.

The bases of $Q_{f, r}^{\ast}$ defined above are $L^{2}$ orthogonal on each triangle
of $\Omega_f$.  Therefore the block of the mass matrix corresponding to contributions
from 
a single macroelement is diagonal on each triangle. This is also observed  in the numerical tests below. However, bases associated to different macroelements can still interfere with each other and the contribution of this overlapping effect to the condition number is reflected in the bound \eqref{ab-bound} with constants $a$ and $b$.

In the numerical tests below we consider the mass matrix 
in the case when the complete mesh is a single vertex macroelement.
Alternatively, this matrix can be  thought of as a "local block matrix" of 
a mass matrix generated by a larger mesh.
On this local mesh we consider all the frame functions generated by 
all the subsimplexes of the mesh.
In all the examples below the local mesh corresponds to the macroelement of 
the origin, and we will consider three cases:
\begin{itemize}
\item Test 1: a macroelement consisting of three triangles with boundary vertices (1, -1), (0, 2), (-1, -1) (left of Figure \ref{fig:m3mesh});
\item Test 2: distortion of Test 1, which has three triangles with boundary vertices (1.3, -0.0001), (0, 2.12), (-1, -0.0001) (right of Figure \ref{fig:m3mesh});
\item Test 3: a macroelement consisting of six triangles with boundary vertices $(2, 0)$, $(1, 1)$, $(0, 1)$, $(-1, 1)$, $(-2, 0)$, $(-1, -1)$, $(1, -1)$  (Figure \ref{fig:m6}). 
\end{itemize}
The diagonal entries of the mass matrix are scaled to one such that the bases are scaled to have unit $L^{2}$ norms. In the tables below, dimension of the frame means the number of functions in the frame representation, while rank of the frame means the rank of the mass matrix. Alternatively, the rank of the frame can be identified as the dimension of 
the space of continuous piecewise polynomial of the same degree on the same mesh.

The condition numbers in Test 1 are shown in Table \ref{tab:m3}. Here $\lambda_{\min}^{\ast}$ denotes the minimal nonzero eigenvalue. Due to the roundoff error, the minimal nonzero eigenvalue slightly decreases as $r$ increases, in contrast to the monotonicity predicted by the min-max principle. Thus the condition numbers also slightly decrease. 
\begin{figure}
\centering 
\includegraphics[width=0.6\linewidth]{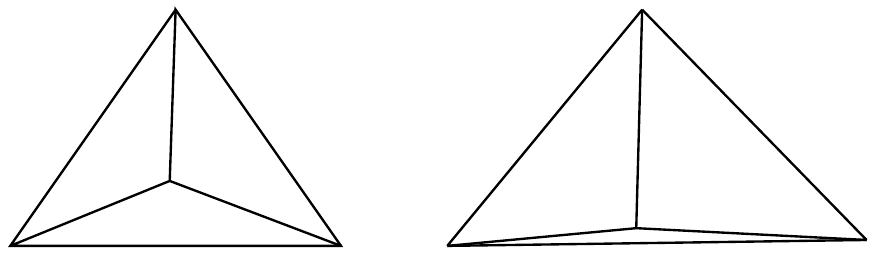}
\caption{Meshes for Test 1 (left) and Test 2 (right). }
\label{fig:m3mesh}
\end{figure}
\begin{table}
\centering
\begin{tabular}{ccccc}\hline
r&3&5&7&9\\\hline
$\lambda_{\max}$&6.623&6.819&6.893& 6.930\\
$\lambda_{\min}^{\ast}$&0.427&0.457&0.472&0.481\\
cond. num. &15.503&14.922&14.600&14.414\\
dim. of frame &33&98&199&336\\
rank of frame &19&46&85&136\\\hline
\end{tabular}
\caption{Results for Test 1. }
\label{tab:m3}
\end{table}
The results on a distorted mesh, Test 2, are shown in Table \ref{tab:m3-distorted}.  The condition numbers are only slightly larger than those in Test 1, showing that the condition numbers are nicely bounded even if there are very thin elements present in the mesh.
\begin{table}
\centering
\begin{tabular}{ccccc}\hline
r&3&5&7&9\\\hline
$\lambda_{\max}$&6.624&6.819&6.893& 6.930\\
$\lambda_{\min}^{\ast}$&0.333&0.383&0.414&0.435\\
cond. num. &19.869&17.809&16.647&15.929\\
dim. of frame &33&98&199&336\\
rank of frame &19&46&85&136\\
\hline
\end{tabular}
\caption{Results for Test 2. }
\label{tab:m3-distorted}
\end{table}

The results for Test 3 are shown in Table \ref{tab:m6}. Similar to Test 1 and 2, the condition number also remains bounded as the degree increases. 
\begin{table}
\centering
\begin{tabular}{ccccc}\hline
r&3&5&7&9\\\hline
$\lambda_{\max}$&6.624&6.819&6.893& 6.930\\
$\lambda_{\min}^{\ast}$&0.348&0.396&0.427&0.446\\
cond. num. &19.049&17.203&16.162&15.525\\
dim. of frame &73&222&455&772\\
rank of frame &43&106&197&316\\
\hline
\end{tabular}
\caption{Results for Test 3. }
\label{tab:m6}
\end{table}



Patterns of the mass matrices, i.e., the zero--nonzero structure, are shown in Figure \ref{fig:m3n3} and  Figure \ref{fig:m3n9} for Test 1.
The pattern only depends on the ordering of the bases and the topology of the mesh. 
Therefore, the results for Test 1 and Test 2 are the same.
In the results below the bases are in the order of the interior vertex, boundary vertices,  boundary edges, interior edges and finally interior modes.
The corresponding results for Test 3 are given in
Figure \ref{fig:m6n3} and Figure \ref{fig:m6n9}. Due  to the locality of the frames, more sparsity appears as the mesh has more triangles.
\begin{figure}
\centering 
\includegraphics[width=0.3\linewidth]{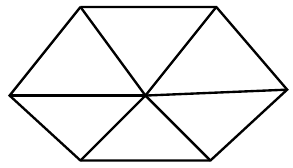}
\caption{Vertex macroelement with 6 triangles. }
\label{fig:m6}
\end{figure}

\begin{figure}[H]\centering
{\setlength{\abovecaptionskip}{-40pt}
\begin{minipage}[t]{.5\textwidth} \vspace{-7pt}
\includegraphics[width=1.08\linewidth]{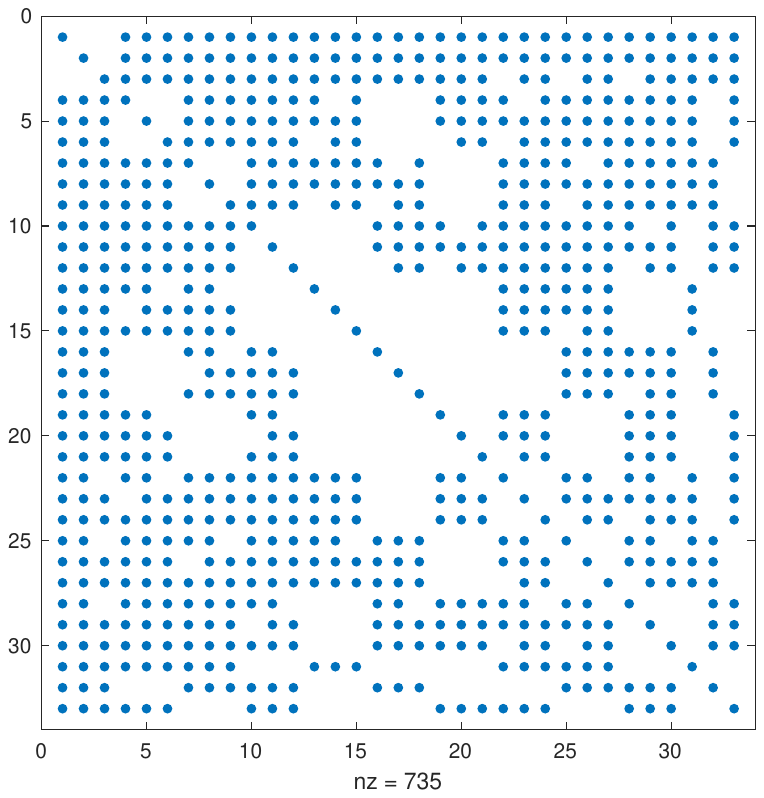}
\caption{Pattern of mass matrix, Test 1, $r=3$.}
\label{fig:m3n3}
\end{minipage}}%
\begin{minipage}[t]{.5\textwidth} \centering \vspace{-10pt}
\includegraphics[width=1.1\linewidth]{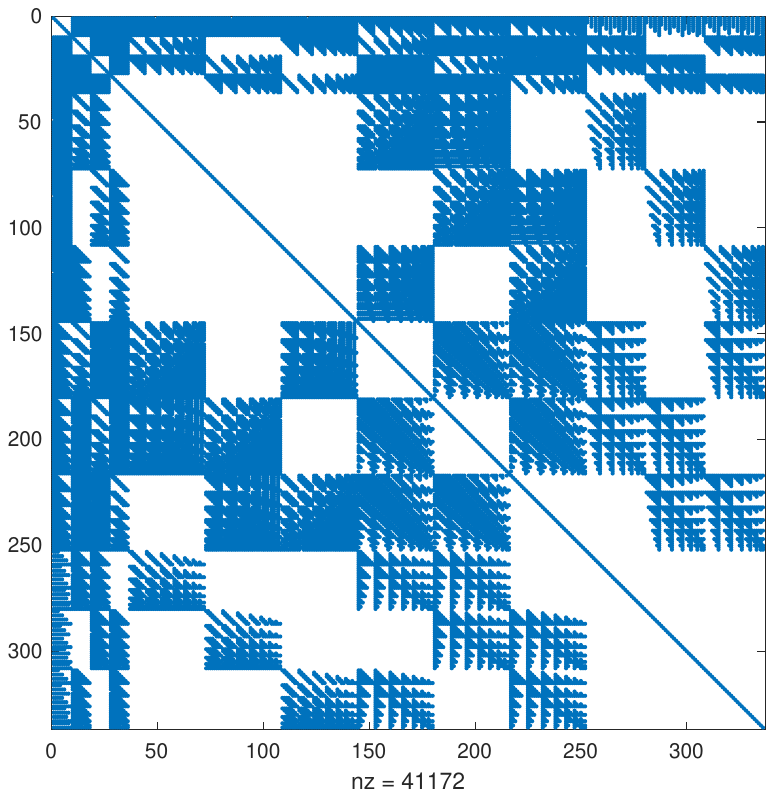}{\setlength{\abovecaptionskip}{-43pt}
\caption{Pattern of mass matrix,  Test 1, $r=9$.}
\label{fig:m3n9}}
\end{minipage}
\end{figure}
\begin{figure}[H]\centering
{\setlength{\abovecaptionskip}{-40pt}
\begin{minipage}[t]{.5\textwidth} \vspace{-7pt}
\includegraphics[width=1.08\linewidth]{./m3n9.pdf}
\caption{Pattern of mass matrix, Test 3,  $r=3$.}
\label{fig:m6n3}
\end{minipage}}%
\begin{minipage}[t]{.5\textwidth} \centering \vspace{-10pt}
\includegraphics[width=1.1\linewidth]{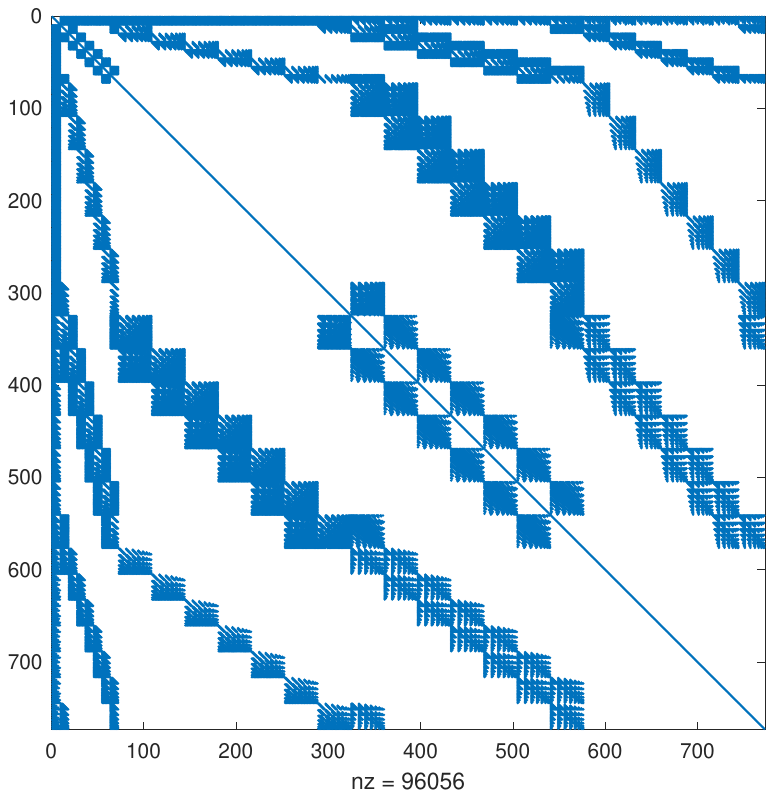}{\setlength{\abovecaptionskip}{-43pt}
\caption{Pattern of mass matrix, Test 3, $r=9$.}
\label{fig:m6n9}}
\end{minipage}
\end{figure}


\section{Iterative methods and preconditioning} \label{precond-it}
We recall the setting discussed in the introduction of this paper, where $V_h \subset H^1(\Omega)$
is a $C^0$ piecewise polynomial  space. 
In the first part of this paper we have outlined how to construct  representations of
the spaces $V_h$ by frames which admit $L^2$ condition numbers which are bounded independently 
of the polynomial degree. The main purpose of this section is to present a brief 
discussion of the use of preconditioned iterative methods in the setting of representations by frames.
In particular, we will clarify how the preconditioned iteration is effected by the boundedness of the 
$L^2$ condition number.
However, before we discuss this in the setting of frames we will first review the more standard  situation when 
the computation relies on a basis for the space $V_h$. 

\subsection{Representation by a basis}
If $\Phi = \{ \phi_j \}_{j=1}^n$ is a basis for the space $V_h$, then the
two bijective maps  
$\tau_{h}: \mathbb{R}^{n}\mapsto V_{h}$ and $\mu_{h}=\tau_{h}^{\ast}: V_{h}^{\ast}\mapsto \mathbb{R}^{n}$,
introduced in Section~\ref{intro},
are used to represent elements of $V_h$ and $V_h^*$, respectively.
The stiffness matrix $\A_h$ admits the representation 
$\mathbb{A}_{h}:=\mu_{h}\mathcal{A}_{h}\tau_{h}$,  where the operator $\mathcal{A}_{h}
: V_h \to V_h^*$ is independent of the choice of basis. Discrete elliptic systems
of the form \eqref{weak-h-1}, or equivalently \eqref{lin-system-1}, are most effectively solved by preconditioned iterative methods.
Therefore,  it seems appropriate to study the effect of the choice of basis on the complete preconditioned system.

In operator form the preconditioned system appears as 
\[
\mathcal{B}_{h} \mathcal{A}_h u_h = \mathcal{B}_{h} f_h,
\]
where the preconditioner $\mathcal{B}_{h}$ is an operator,
$\mathcal{B}_{h}: V_{h}^* \to V_{h}$, which is symmetric and positive definite with respect to the duality pairing between $V_h^*$ and $V_h$. Furthermore, its standard representation is the matrix 
${{\B}_h} = \tau_h^{-1} {\mathcal{B}_h} \mu_h^{-1} = \tau_h^{-1} {\mathcal{B}_h} \tau_h^{-*}$.
The two basic necessary conditions for the construction of an effective  preconditioner are,
{\it i)} the spectral condition number of ${\mathcal{B}_h}{\mathcal{A}_h}$ is well behaved, and {\it ii)}
the matrix-vector products of the form ${{\B}_h}{{\A}_h}c$, for any $c \in \R^n$, can 
be evaluated  fast. Since the stiffness matrix usually is represented by a sparse matrix, the second condition  will hold if the matrix-vector products of the form
${{\B}_h}c$ can be evaluated fast.
It is not the  purpose of this paper to discuss the design of effective preconditioners. 
Instead we refer to \cite{mardal2011preconditioning, xu1992iterative, xu2016algebraic} for general discussions of such constructions.
Our main concern here is to discuss how the $L^2$ condition number of the  basis influences  the key properties of the 
preconditioned iterative method.

The convergence of a standard iterative method for the preconditioned system, such as the conjugate gradient method, will be governed by the spectral condition number of the 
coefficient matrix, 
\begin{equation}\label{similar}
{{\B}_h} {{\A}_h} = \tau_h^{-1} {\mathcal{B}_h} {\mathcal{A}_h} \tau_h. 
\end{equation}
However, this matrix is similar to the operator 
${\mathcal{B}_h}{\mathcal{A}_h}$, and therefore its condition number is independent of the  basis.
On the other hand, the basis will effect the properties of the two matrices
$\A_h$ and $\B_h$. These operators have to be evaluated in each iteration, and their 
conditioning will effect the numerical stability of the computations. 
Recall the inequality \eqref{prod-cond}, which relates $\kappa(\A_h)$, $\kappa(\mathcal I_h \mathcal A_h)$ and the $L^2$ condition number of the basis, $\kappa(\M_h)$.
In fact, $\kappa(\M_h)$ is the only quantity  on the right hand side of the inequality 
\eqref{prod-cond} which is basis dependent. Furthermore, the numerical experiments presented in 
Section~\ref{sec:notation} indicate that this inequality is rather sharp.
Therefore, if the mass matrix is well-conditioned then 
$\kappa(\A_h)$ will behave approximately like the basis independent condition number 
$\kappa(\mathcal I_h \mathcal A_h)$. This condition number reflects the properties of the elliptic operator 
we are approximating. 

The  situation for the preconditioner is similar. The matrix $\B_h$ admits the representation 
\begin{equation}\label{B-rep}
\B_h = \tau_h^{-1} {\mathcal{B}_h} \mathcal I_h^{-1} \mathcal I_h \tau_h^{-*}
= (\tau_h^{-1} {\mathcal{B}_h} \mathcal I_h^{-1} \tau_h) \M_h^{-1},
\end{equation}
and from this we can easily derive  the inequality
\begin{equation}\label{prod-cond-B}
\kappa(\B_h) \le \kappa({\mathcal{B}_h} {\mathcal{I}_h}^{-1}) \kappa(\M_h).
\end{equation}
Therefore, if the mass matrix is well conditioned, then we can conclude that $\kappa(\B_h)$ is essentially bounded by a
basis independent quantity. 
We have therefore seen that, even if the choice of basis has no direct effect on the conditioning of the 
preconditioned system, an $L^2$ well-conditioned basis will result in matrix representations $\A_h$ and $\B_h$ with 
condition numbers that roughly behave like their basis independent counterparts.
Below we will argue that these conclusions also hold if we allow representations 
by frames.

\subsection{Representations by frames}\label{solvebyframe}
Assume that we are given a frame $\Phi = \{\phi_{j} \}_{j=1}^N$ of $V_h$, where in general $N$ is larger than the dimension of $V_{h}$. The operators $\tau_{h}: \mathbb{R}^{N}\mapsto V_{h}$ and $\mu_{h}: V_{h}^{\ast}\mapsto \mathbb{R}^{N}$, are defined as above, i.e., 
\[
c \mapsto \tau_h(c) = \sum_j c_j \phi_j,  \quad \text{and} \quad \mu(f)_i  = \< f ,\phi_i \>,
\]
such that the identity $\mu_h f \cdot c =\< f, \tau_h(c) \>$ holds. In this setting the operator $\tau_h$
is surjective, the operator $\mu_{h}: V_{h}^{\ast}\mapsto \mathbb{R}^{N}$ is injective, and 
as above $\tau_h$ and $\mu_h$ correspond to the dual of the other. 
If the $L^2$ condition number of the frame, $\mathcal{K}(\Phi)$, is controlled then we will argue that 
also in this case the matrix representations of the discrete elliptic operator $\mathcal A_h$ and suitable preconditioners $\mathcal B_h$ behaves roughly like 
the corresponding basis independent counterparts. In fact, this simply follows by proper block decompositions of the matrices.

The stiffness matrix,  representing  the coefficient operator $\mathcal{A}_{h}: V_{h}\mapsto V_{h}^{\ast}$ is still defined as $\mathbb{A}_{h}=\mu_{h}\mathcal{A}_{h}\tau_{h}$, cf. \eqref{A-diagram}.
While the operator $\mathcal A_h$ is positive definite, the stiffness matrix $\mathbb{A}_h$ is only positive 
semi-definite with   $\ker (\mathbb{A}_{h}) =\ker \left (\tau_{h}\right)$ and 
\[
\mathrm{Im}(\A_h) = \mathrm{Im}(\mu_h) = \ker \left (\tau_{h}\right)^{\perp}.
\]
Here the orthogonal complement is with respect to the standard Euclidean inner product of $\R^N$.
In fact, with respect to the orthogonal decomposition $\mathrm{Im}(\mu_h) \oplus \ker \left (\tau_{h}\right)$
 the matrix $\A_h$ has a block structure of the form 
\[
\left (
\begin{array}{cc}
\tilde{\A}_{h} & 0\\
0 &  0
\end{array}
\right ),
\]
where the matrix $\tilde{\A}_{h}$ is a positive definite matrix on $\mathrm{Im}(\mu_h)$.
The mass matrix, $\M_h$,
with elements $\<\phi_i, \phi_j \>$, has a similar block structure of the form 
$$
 \mathbb{M}_{h}=\left ( 
 \begin{array}{cc}
 \tilde{\mathbb{M}}_{h} & 0\\
 0 & 0
 \end{array}
 \right )
 $$
with respect to the decomposition $\mathrm{Im}(\mu_h) \oplus \ker \left (\tau_{h}\right)$. 
Here the 
the matrix $\tilde{\M}_{h}$, mapping  $\mathrm{Im}(\mu_h)$ into itself, is positive definite.
Furthermore, it follows from the observation done in Remark~\ref{condition-number} that  
$\mathcal{K}(\Phi) = \kappa(\tilde{\mathbb{M}}_{h})$. 
The marices $\tilde{\A}_{h} $ and $\tilde{\M}_{h}$ can be related by the identity
\[
\tilde{\A}_{h} = \tilde{\M}_{h} (\tilde \tau_h^{-1} \mathcal I_h \mathcal A_h \tilde \tau_h),
\]
where $\tilde \tau_h$ denotes the restriction of $\tau_h$ to $\mathrm{Im}(\mu_h)$. By arguing as above, this leads to 
\begin{equation}\label{prod-cond-tilde}
\kappa(\tilde \A_h) \le \kappa(\mathcal I_h \mathcal A_h) \kappa(\tilde \M_h),
\end{equation}
which is a generalization of  inequality  \eqref{prod-cond}.
Therefore, we can again conclude that $\kappa(\tilde A_h)$ behaves roughly like its 
basis independent counterpart, $\kappa(\mathcal I_h \mathcal A_h)$, as long as the frame condition number $\mathcal{K}(\Phi) = \kappa(\tilde{\mathbb{M}}_{h})$
is well controlled.

In the setting of frames a preconditioner  is represented by an $N \times N$ matrix 
$\mathbb{B}_{h}$ 
which is symmetric and positive definite with respect to the Euclidean inner product. Furthermore,
the corresponding operator $\mathcal{B}_h : V_h^* \to V_h$ is given by the diagram 
\[
 \begin{diagram}
\mathbb{R}^{N} & \rTo^{\mathbb{B}_{h}} &\mathbb{R}^{N}\\
       \uTo^{\mu_{h}}       &  & \dTo^{\tau_{h}}\\
 V_{h}^{\ast} & \rTo^{\mathcal{B}_{h}} &V_{h}
 \end{diagram}
 \]
i.e., $\mathcal{B}_{h}:=\tau_{h}\mathbb{B}_{h}\mu_{h}$. The operator $\mathcal{B}_{h}$ is symmetric in the sense that 
for $f,g \in V_h^*$,
\[
\< f, \mathcal{B}_h g \> = \mu_h(f) \cdot \B_h \mu_h(g) = \< g, \mathcal{B}_h f \>,
\]
and $\mathcal{B}_h$ is positive definite since $\B_h$ has this property.
Furthermore, the identity 
\begin{align}\label{frame-BA}
\tau_{h}\mathbb{B}_{h}\mathbb{A}_{h}=\mathcal{B}_{h}\mathcal{A}_{h}\tau_{h}
\end{align}
holds.
We will implicitly assume that $\mathcal{B}_h$ 
has an interpretation as an operator from $V_h^*$ to $V_h$ which is independent of the choice of frame.
In this respect, also the operator $\mathcal{B}_h\mathcal{A}_h$, and its spectral condition number, will be 
frame independent. In fact, we will justify this assumption in Section~\ref{rep-pre} below, in the case when the frame is derived 
from a basis of each of  the local spaces  ${Q_{f, r}^{\ast}}$, cf. \eqref{bubble-decomposition}.

Consider a linear system of the form \eqref{weak-h-1}, i.e., $\mathcal{A}_h u = f$, where the data $f \in V_h^*$
and $u \in V_h$. The preconditioned version of this system takes the form 
\begin{align}\label{operator-BA}
\mathcal{B}_{h}\mathcal{A}_{h}u=\mathcal{B}_{h}f,
\end{align}
or in matrix-vector form
 \begin{align}\label{matrix-BA}
\mathbb{B}_{h}\mathbb{A}_{h}c=\mathbb{B}_{h}\mu_h(f),
\end{align}
where $c \in \R^N$ is any vector such that $\tau_h(c) = u$. Hence, even if the solution $u \in V_h$ is uniquely determined by the system, the vector $c$ is only determined up to addition of elements in 
$\ker \left (\tau_{h}\right) = \ker \left (\B_h \A_h \right)$.

It is well known that Krylov subspace methods can be used to solve semidefinite systems,
cf. for example \cite{lee2006convergence}. 
In fact, if we consider a system of the form \eqref{matrix-BA}, with initial guess $c^0 = 0$, then 
the initial residual 
\[
r^0 = \B_h \mu_h(f) \in \mathrm{Im}(\B_{h}\A_{h})= \ker(\B_{h}\A_{h})^{\angle}=\ker(\tau_h)^{\angle},
\]
where the superscript $\angle$ indicates orthogonality with respect to the  inner product generated by the positive definite matrix $\mathbb{B}_{h}^{-1}$.
The matrix $\B_{h}\A_{h}$ maps $\ker(\tau_h)^{\angle}$ to itself, and as consequence all the vectors in the associated Krylov spaces, spanned by vectors 
of the form $(\B_{h}\A_{h})^{k}r^0$, will be in this space.
As a consequence, we can view the iterative method as if it is  restricted to the space $\ker(\tau_h)^{\angle}$, where the system 
\eqref{matrix-BA} has a unique solution. Furthermore, the convergence rate will be bounded by the spectral condition number of 
the coefficient matrix  $\B_{h}\A_{h}$ restricted to this space.
If we let $\hat{\tau}_{h}: \ker(\tau_{h})^{\angle}\mapsto V_{h}$ be the restriction of $\tau_{h}$  to
$\ker(\tau_{h})^{\angle}$ then $\hat{\tau}$ is invertible and from \eqref{frame-BA} we obtain 
\[
{\B}_{h}{\A}_{h}|_{\ker(\tau_{h})^{\angle}} =\hat{\tau}_{h}^{-1}\mathcal{B}_{h}\mathcal{A}_{h}\hat{\tau}_{h}
\]
as a generalization of the similarity relation \eqref{similar}.
In particular, this shows that the spectral condition number of the matrix $\B_h\A_h$, restricted to 
$\ker(\tau_{h})^{\angle}$,
is equal to the spectral condition number of the operator $\mathcal{B}_h \mathcal{A}_h : V_h \to V_h$.
Therefore, also in the case of representations by frames, we can  conclude that the performance of a preconditioned Krylov space method
is, in a proper sense, independent of the choice of representation of the spaces $V_h$ and $V_h^*$.
However, as we have seen above, a well condition frame guarantees that the conditioning  of the stiffness matrix reflects 
the condition number of its basis independent counterpart, cf. \eqref{prod-cond-tilde}.
A similar conclusion for the preconditioner, i.e., a generalization of inequality \eqref{prod-cond-B}, is also easily established.

%
%
%
\subsection{Representation of the preconditioner}\label{rep-pre}
The discussion above is based on the assumption that the matrix $\B_h$, representing the preconditioner 
$\mathcal{B}_h : V_h^* \to V_h$, is positive definite, and that the operator
$\mathcal{B}_h$ has an interpretation which is independent of the representations of the spaces $V_h$ and $V_h^*$.
Here we will argue that if $V_h$ is of the form ${\mathcal{P}_r}(\mathcal{T}_{h})$ and 
the frame $\Phi$  is generated by an overlapping decomposition of the form 
\eqref{bubble-decomposition}, then this assumption is indeed very natural. To illustrate this  
we consider an additive Schwartz preconditioner of the form proposed in \cite{schoberl2008additive}.
In operator form this preconditioner has the structure
\[
\mathcal{B}_h  = \mathcal{B}_h^0 +  \sum_{f \in \Delta} \mathcal{B}_{f,h},
\]
where $\mathcal{B}_h^0$ is a ``coarse space preconditioner",  and $\mathcal{B}_{f,h}$ 
are local preconditioners defined with respect to local spaces ${Q_{f, r}^{\ast}} \subset V_h$.
More precisely, each of the operators $\mathcal{B}_{f,h}$ are preconditioner for the corresponding local operator 
$\mathcal{A}_{f,h}$ defined by the bilinear form $a(\cdot,\cdot)$ with respect to the local space ${Q_{f, r}^{\ast}}$,
while $\mathcal{B}_h^0$ is a corresponding preconditioner defined with respect to the piecewise linear space
${\mathcal{P}_1}(\mathcal{T}_{h}) \subset V_h$.

Consider the set up in Section~\ref{bubble-frame} above, where for each $f \in \Delta = \Delta(\mathcal{T}_h)$
the set $\Phi_f = \{ \phi_{f,k} \}_{k=1}^{N_f}$ is a basis for 
the local space ${Q_{f, r}^{\ast}}$, and $\Phi = \{\Phi_f \}_{f \in \Delta}$.
The 
natural representations of the operators $\mathcal{B}_{f,h}$ are $N_f \times N_f$ matrices of the form $\B_{f,h}= \tau_{f,h}^{-1} \mathcal{B}_{f,h} \mu_{f,h}^{-1}$,
where the representations $\tau_{f,h} $ and $\mu_{f,h}$ are defined as above, but now with respect to the local spaces
 ${Q_{f, r}^{\ast}}$. Furthermore, since $\Phi_f$ is a basis for this space, each of the maps $\tau_{f,h}$ and $\mu_{f,h}$ is invertible,
 and as a consequence, the matrix $\B_{f,h}$ is positive definite.
 The matrix $\B_h$, representing the preconditioner $\mathcal{B}_h$, will now be of the form  
 \[
 \B_h = \B_h^0 + 
 \diag\{\B_{f,h}\}_{f \in \Delta} : \R^N \to \R^N,
 \]
 where $N = \sum_f N_f$. Here $\B_h^0$ is a symmetric and positive semidefinite matrix representing 
 the operator $\mathcal{B}_h^0$, while the block diagonal matrix, $\diag\{\B_{f,h}\}$, is symmetric and positive definite, 
 since each block has this property. We refer to \cite{schoberl2008additive} for more details and tests of numerical performance.
 
 {
 \begin{remark}\label{h-dependence}
 The additional global operator $\mathcal{B}_h^0$ appears in the preconditioner proposed in \cite{schoberl2008additive}
 in order to obtain condition numbers which are independent of the mesh size $h$. 
 In the present paper, where we consider the mesh $\mathcal{T}_h$  to be fixed,
we could have dropped this term and still obtain condition numbers which are bounded uniformly with respect to the polynomial degree $r$.
 \end{remark}
 }

\section{Concluding remarks}\label{sec:concluding}
The purpose of this paper is to discuss how to represent $H^1$ finite element spaces of high degree. More precisely, we study the spaces 
$\mathcal{P}_r(\mathcal{T}_h)$, consisting of $C^0$ piecewise polynomial spaces 
of degree $r$ with respect to a simplicial mesh $\mathcal{T}_h$. 
When the degree $r$ grows, the choice of basis for the spaces 
will intuitively affect the properties of the corresponding linear systems derived from a finite element discretization. 
The construction outlined in this paper, based on properties of the bubble transform and of Jacobi polynomials with respect to 
simplexes, leads to frames with $L^2$ condition numbers which are independent of the polynomial degree. 
In this respect, we have been able to present a procedure to
construct  {\it well-conditioned frames.}
Furthermore, we have shown that such frames leads to stiffness matrices, and matrix representations of 
corresponding preconditioners, with condition numbers that behave roughly like their operator counterparts.

%

A possible disadvantage of a representation by a frame instead of a basis, is that 
the number of unknowns is increasing. For the frames proposed above the following table compares the dimension of the frame with a corresponding standard basis.
 {
 \center \footnotesize
\begin{tabular}{|c|c|c|}
\hline
& Basis & Frame \\\hline
1D & $V+(r-1)E=rV-(r-1)$ & $V+rV+(r-1)E=(2r-1)V-(r-1)$\\
2D & $V+(r-1)E+\left (\frac{1}{2}(r+2)(r+1)-3 \right )F$ & $(r+1)V+\frac{1}{2}r(r-1)E+\left (\frac{1}{2}(r+2)(r+1)-3 \right )F$ \\
3D & $V+(r-1)E+\left (\frac{1}{2}(r+2)(r+1)-3 \right )F $& $(r+1)V+\frac{1}{2}r(r-1)E+\frac{1}{6}r(r-1)(r-2)F$\\
     &  $\quad +\left (\frac{1}{6}(r+3)(r+2)(r+1)-4 \right )T $ &  $\quad +\left (\frac{1}{6}(r+3)(r+2)(r+1)-4 \right )T$\\\hline
\end{tabular}
\\
}
\bigskip
Here $V, E, F$ and $T$ are the number of vertices, edges, faces and 3D cells 
in the triangulation. We observe that even if 
the frame representations have redundancy,  the asymptotic order of the total dimensions remains the same. Of course, in addition to condition numbers there are 
potentially a number of other criteria that could have been used to choose a basis or a frame, for example sparsity and fast evaluation of the stiffness matrix and the preconditioner. However, such discussions are outside the scope of the present paper.

\section*{Appendix. Jacobi polynomials on simplexes}  
The purpose of this appendix is to present precise formulas for basis functions of the local spaces $Q_{f,r}^*$.
We will first consider the case when $\dim f <d$.  We recall from Section~\ref{sec:H1basis} above that if we can construct bases $\{ p_{\bm{s}} \}$ for the polynomial spaces $\mathcal{P}_{r-m-1}(S_m^c)$, which are orthonormal with respect to 
the weighted $L^2$ inner product with weight $w_m(\lambda) = (\Pi\lambda)_m^2b(\lambda)^{d-m-1}$,
then the corresponding set $\{ \lambda_f^* [(\Pi\lambda)_m p_{\bm{s}}] \}$ will be uniformly scaled bases for the spaces $Q^{\ast}_{f, r}$. Therefore, we will briefly outline how such bases for the spaces 
$\mathcal{P}_{r-m-1}(S_m^c)$ can be constructed. This discussion is based on multi-variate Jacobi polynomials on simplexes,
and is mostly taken from 
\cite[Section 5.3]{dunkl2014orthogonal}, cf. also \cite{xin2011well} where the scaling of the polynomials is corrected.

We use $J_{s}^{\alpha,\beta}(x)$ to denote the orthonormal Jacobi polynomial on interval $[-1, 1]$ with weight $w^{\alpha,\beta}:=(1-x)^{\alpha}(1+x)^{\beta}$, i.e.
$$
\int_{-1}^{1}J_{s}^{\alpha, \beta}(x)J_{s'}^{\alpha,\beta}(x)(1-x)^{\alpha}(1+x)^{\beta}dx=\delta_{ss'}.
$$
Using a linear transform $\xi=(x+1)/2$, we get Jacobi polynomials on the reference interval $[0,1]$:
$$
\int_{0}^{1}J_{s}^{\alpha, \beta}(2\xi -1)J_{s'}^{\alpha,\beta}(2\xi -1)(1-\xi)^{\alpha}\xi^{\beta}d\xi=c_{s}^{\alpha,\beta}\delta_{ss'},
$$
where 
\begin{align}\label{cn}
c_{s}^{\alpha,\beta}=2^{-\alpha-\beta-1}.
\end{align}
We introduce the notation
$$
\tilde{J}_{s}^{\alpha, \beta}(\xi):=J^{\alpha, \beta}_{s}(2\xi -1),
$$
so that $\{\tilde{J}_{s}^{\alpha, \beta}(\xi)\}, s=0, 1,  \cdots$ is a single variate orthogonal basis on $[0, 1]$ with weight $w^{\alpha,\beta}$.

We define
$$
b_{j}(\lambda):=1-\sum_{i=0}^{j-1}\lambda_{i}.
$$

Given integer $s\geq 0$ and $0\leq m< d$, we will define a basis for the space $\mathcal{P}_{\bm{s}}\left (S_{m}^{c}\right )$. We consider the  polynomials $J_{\bm{s}}(\lambda)=J_{\bm{s}}(\lambda_{0}, \lambda_{1}, \cdots, \lambda_{m})$ given by
\begin{align}\label{nd-orthogonal0}
J_{\bm{s}}(\lambda):=c_{\bm{s}}^{-1}\prod_{j=0}^{m}b_{j}(\lambda)^{s_{j}}\tilde{J}_{s_{j}}^{a_{j}, 2}\left(  \frac{\lambda_{j}}{b_{j}(\lambda)}\right ),\quad |\bm{s}|\leq s,
\end{align}
where $\bm{s}=(s_{0}, \dots, s_{m})$ is a multi-index. Here the constants $a_{j}$ and $c_{\bm{s}}$ are given by   $a_{j}=2\sum_{i=j+1}^{m}s_{i}+d+2m-3j-1$ and
$$
c_{\bm{s}}^{-2}=\Pi_{j=0}^{m} 2^{a_{j}+3}=\Pi_{j=0}^{m} 2^{2\sum_{i=j+1}^{m}s_{i}+d+2m-3j+2}.
$$
The polynomials \eqref{nd-orthogonal0} form a mutually orthonormal basis for $\mathcal{P}_{s}(S_{m}^{c})$ with the weight $w_{m}$.

For $f\in \Delta_{d}$, the construction is similar and more standard. To construct a basis for $Q_{f, r}^{\ast}=\0{\mathcal{P}}_{r}(f)$,  we can follow \cite{dunkl2014orthogonal}  to construct $L^{2}$ orthonormal bases $\left\{q_{\bm{s}}\right\}$ for $\mathcal{P}_{r-d-1}(S_{d})$ with weight $(\Pi\lambda)_{d}^{2}$. The  set $\{ \lambda_f^* [(\Pi\lambda)_d q_{\bm{s}}] \}$ will be uniformly scaled bases for the spaces $Q^{\ast}_{f, r}$, where $\lambda_{f}: f\mapsto S_{d}$ is defined by the barycentric coordinates.



\bigskip

\noindent {\bf Acknowledgments.}

The authors are grateful to  Douglas N. Arnold, Richard S. Falk and Jinchao Xu for several discussions about the results of this paper.

The research leading to these results has received funding from the  European Research Council under the European Union's Seventh Framework Programme (FP7/2007-2013) / ERC grant agreement 339643. The research of the first author leading to the results of this paper was partly carried out during his affiliation with the University of Oslo, partly supported by China Scholarship Council (CSC), project 201506010013.


\bibliographystyle{plain}       
\bibliography{hpbib}{}   
\end{document}